\documentclass[10pt,oneside]{amsart}
\usepackage{graphicx}
\usepackage[dvipsnames,usenames]{xcolor}
\usepackage[colorlinks=true,urlcolor=ForestGreen,linkcolor=ForestGreen,citecolor=ForestGreen]{hyperref}
\hypersetup{
	linkbordercolor={1 0 0}, 
	citebordercolor={0 1 0} 
}
\usepackage{mathrsfs}
\usepackage[UKenglish]{babel}
\usepackage{amssymb,amsxtra}
\usepackage{enumerate}
\usepackage{bbm}
\usepackage{relsize}
\usepackage{tikz}

\usepackage{amscd}

\usepackage{xcolor}
\usepackage[all]{xy}

\usetikzlibrary{trees,decorations.pathmorphing,decorations.markings,matrix,shapes}

\iftrue
\makeatletter
\def\@settitle{%
	\vspace*{-20pt}
	\begin{flushleft}%
		\baselineskip14\p@\relax
		\normalfont\bfseries\LARGE
		\@title
	\end{flushleft}%
}
\def\@setauthors{%
	\begingroup
	\def\thanks{\protect\thanks@warning}%
	\trivlist
	\large \@topsep30\p@\relax
	\advance\@topsep by -\baselineskip
	\item\relax
	\author@andify\authors
	\def\\{\protect\linebreak}%
	\authors
	\ifx\@empty\contribs
	\else
	,\penalty-3 \space \@setcontribs
	\@closetoccontribs
	\fi
	\normalfont
	\@setaddresses
	\endtrivlist
	\endgroup
}
\def\@setaddresses{\par
	\nobreak \begingroup\raggedright
	\small
	\def\author##1{\nobreak\addvspace\smallskipamount}%
	\def\\{\unskip, \ignorespaces}%
	\interlinepenalty\@M
	\def\address##1##2{\begingroup
		\par\addvspace\bigskipamount\noindent
		\@ifnotempty{##1}{(\ignorespaces##1\unskip) }%
		{\ignorespaces##2}\par\endgroup}%
	\def\curraddr##1##2{\begingroup
		\@ifnotempty{##2}{\nobreak\noindent\curraddrname
			\@ifnotempty{##1}{, \ignorespaces##1\unskip}\/:\space
			##2\par}\endgroup}%
	\def\email##1##2{\begingroup
		\@ifnotempty{##2}{\smallskip\nobreak\noindent E-mail address%
			\@ifnotempty{##1}{, \ignorespaces##1\unskip}\/:\space
			\ttfamily##2\par}\endgroup}%
	\def\urladdr##1##2{\begingroup
		\def~{\char`\~}%
		\@ifnotempty{##2}{\nobreak\noindent\urladdrname
			\@ifnotempty{##1}{, \ignorespaces##1\unskip}\/:\space
			\ttfamily##2\par}\endgroup}%
	\addresses
	\endgroup
	\global\let\addresses=\@empty
}
\def\@setabstracta{%
	\ifvoid\abstractbox
	\else
	\skip@25\p@ \advance\skip@-\lastskip
	\advance\skip@-\baselineskip \vskip\skip@
	\box\abstractbox
	\prevdepth\z@ 
	\vskip-15pt
	\fi
}

\def\ps@headings
{\ps@empty
	\def\@evenhead{%
	  \setTrue{runhead}%
		\normalfont\scriptsize
		\rlap{\thepage}\hfill
		\def\thanks{\protect\thanks@warning}%
		\leftmark{}{}}%
	\def\@oddhead{%
		\setTrue{runhead}%
		\normalfont\scriptsize
		\def\thanks{\protect\thanks@warning}%
\rightmark{}{}\hfill \llap{\thepage}}%
\let\@mkboth\markboth
}\ps@headings

\def\section{\@startsection{section}{1}%
	\z@{-1.2\linespacing\@plus-.5\linespacing}{.8\linespacing}%
	{\normalfont\bfseries\Large}}
\def\subsection{\@startsection{subsection}{2}%
	\z@{-.8\linespacing\@plus-.3\linespacing}{.3\linespacing\@plus.2\linespacing}%
	{\normalfont\bfseries\large}}
\def\subsubsection{\@startsection{subsubsection}{3}%
	\z@{.7\linespacing\@plus.1\linespacing}{-1.5ex}%
	{\normalfont\bfseries}}
\def\@secnumfont{\bfseries}
\makeatother
\fi 

\makeatother

\newtheorem{theorem}{Theorem}[section]
\newtheorem*{theorem*}{Theorem}

\newtheorem{proposition}[theorem]{Proposition}
\newtheorem{fact}[theorem]{Fact}

\newtheorem{mth}[theorem]{Main Theorem}

\theoremstyle{definition}
\newtheorem{definition}[theorem]{Definition}

\newtheorem{example}[theorem]{Example}
\newtheorem{remark}[theorem]{Remark}

\numberwithin{equation}{section}

\newcommand{\tart}[2]{\begin{tabular}{c} #1 \\ #2\end{tabular}}
\newcommand{\incg}[2]{\includegraphics[height=#1em]{#2.eps}}
\newcommand{\targ}[3]{\tart{\incg{#1}{#2}}{#3}}

\newcommand{\CC}{\mathcal{C}}

\newcommand{\MM}{\mathcal{M}}

\newcommand{\gr}{\mathrm{gr}}

\newcommand{\CubeFlowCat}{\mathscr{C}_C}
\newcommand{\Moduli}{\mathcal{M}}
\newcommand{\vect}{\overline}
\newcommand{\sm}{\setminus}
\newcommand{\gen}[1]{#1}

\newcommand{\np}{\pagebreak}
\newcommand{\vs}{\vskip10mm}
\newcommand{\bs}{\smallbreak}
\newcommand{\bb}{\bigbreak}
\newcommand{\h}{\noindent}
\newcommand{\R}{\mathbb R}
\newcommand{\N}{\mathbb N}
\newcommand{\E}{\mathbb E}
\newcommand{\x}{\times}
\newcommand{\Hom}{\mathrm{Hom}}

\newcommand{\Z}{\mathbb{Z}}

\newcommand{\ind}{\mathrm{ind}}

\begin{document}
	
\pagestyle{plain}

{\huge

\h{\bf
Khovanov-Lipshitz-Sarkar homotopy type for links in \\
thickened surfaces}
\vs

\h{Louis H. Kauffman
, Igor Mikhailovich Nikonov,
and
Eiji Ogasa}
\vs

\h{\bf Abstract.}  
We define a Khovanov-Lipshitz-Sarkar stable homotopy type for the homotopical Khovanov homology of links in the thickened torus 
%
after the authors introduced 
that in the case of higher genus surfaces 
in the previous paper of this one. 
\vs
\tableofcontents

\vs
\section{Introduction}\label{intro}

\h
In this paper, a surface means a closed oriented surface unless otherwise stated.
Of course, a surface may or may not be the sphere.
We discuss links in thickened surfaces.
If $\mathcal L$ is a link in a thickened surface,
then a link diagram $L$ which represents  $\mathcal L$  lies in the surface.
 Since our theory has a special behavior at genus one, in this paper a higher genus surface means a surface with genus greater than one unless otherwise stated.
 In the previous paper~\cite{KauffmanNikonovOgasa} of this one,
  the authors discussed the higher genus case.
 In the present paper, we mainly discuss the torus case.
\\

Let $\mathcal K$ be a link in the thickened torus.
Let $K$ be a link diagram in the torus which represents $\mathcal K$.
Call a poset associated with a decorated Kauffman state, a {\it dposet}.
See \cite{KauffmanNikonovOgasa, LSk} for
decorated Kauffman states, or decorated resolution configurations.
Dposets are defined for all pairs of enhanced Kauffman states.

We discuss the following case, which will be introduced in \S\ref{mod}.
We choose the right pair or the left one for the ladybug Kauffman state
(see \cite{KauffmanNikonovOgasa, LSk}).
We determine a degree 1 homology class $\lambda$ of $T^2$.
After that,
we define a cubic moduli for any dposet of $K$,
and construct
{\it Khovanov-Lipshitz-Sarkar stable homotopy type}
for the homotopical Khovanov chain complex 
(\cite{KauffmanNikonovOgasa, MN}) of $K$.
We define
{\it Khovanov-Lipshitz-Sarkar stable homotopy type}
for $\mathcal K$ to be that for $K$.

Make
the set of all Khovanov-Lipshitz-Sarkar stable homotopy types for all  $\lambda$ and for a fixed choice of the right and the left. It
is a link type invariant. 
There are infinitely many  $\lambda$ but there are finite numbers of stable homotopy types.
Recall in \cite{KauffmanNikonovOgasa} that in the higher genus case,
we give only one stable homotopy type for any link diagram
after we choose the right pair or the left one,
and therefore, for any link type.

We prove that
the set of our Khovanov-Lipshitz-Sarkar stable homotopy types
is stronger than
the homotopical  Khovanov homology of $\mathcal K$. 
 It is a meaningful Khovanov stable homotopy type of links in
a 3-manifold other than the 3-sphere.


\begin{mth}\label{main}
$(1)$
We define
 Khovanov-Lipshitz-Sarkar stable homotopy type
 for
 `links in the thickened torus, a degree 1 homology class $\lambda$, and
 a choice of the right and the left pair'.

\bb\h $(2)$
Our new invariants $($the stable homotopy type$)$ in $(1)$ above 
gives an invariant stronger than
the  homotopical Khovanov  homology
as invariants of links in the thickened torus.
We use the second Steenrod square to prove it.
\end{mth}
\bb

\subsection{Background} 
Jones 
\cite{Jones} 
discovered the Jones polynomial of links in the 3-sphere.
Kauffman 
\cite{Kauffmanstate} 
made a very short alternative proof of 
\cite{Jones} when he invented Kauffman states.
Khovanov 
\cite{K} introduced 
Khovanov homology  of links in the 3-sphere 
when making enhanced Kauffman states from Kauffman states.
 Bar-Natan \cite{B}
proved that 
Khovanov homology  of links in the 3-sphere 
is a topological invariant stronger than the Jones polynomial of those.
Asaeda, Przytycki, and Sikora \cite{APS} 
extended 
Khovanov homology to the case of thickened surfaces.
Manturov and Nikonov \cite{MN}
made an alternative definition of \cite{APS} in the $\Z_2$ coefficient case:
they call it {\it homotopical Khovanov homology}.
Kauffman, Nikonov and Ogasa \cite{KauffmanNikonovOgasa} 
extended \cite{MN} to the $\Z$ coefficient case. 

Lipshitz and Sarkar 
\cite{LSk} 
invented a consistent construction of 
a stable homotopy type, 
of a CW complex for any given link $L$ in the 3-sphere 
whose homology is Khovanov homology: 
It is called 
{\it 
Khovanov-Lipshitz-Sarkar stable homotopy type of links in $S^3$}.   
Lipshitz and Sarkar 
\cite{LSs} 
introduced how to calculate the second Steenrod square of Khovanov-Lipshitz-Sarkar stable homotopy type of links in $S^3$. 
Seed 
\cite{Seed} 
calculated it by their method on computer and proved 
that the second Steenrod square of 
Khovanov-Lipshitz-Sarkar stable homotopy type of links in $S^3$
is stronger than Khovanov homology of links in $S^3$. 
Kauffman, Nikonov and Ogasa \cite{KauffmanNikonovOgasa} 
extended 
Khovanov-Lipshitz-Sarkar stable homotopy type 
to the case of thickened higher genus case. 
\\

\h Remark: 
We have open questions 
whether we can extend the Jones polynomial, 
Khovanov homology, and 
Khovanov-Lipshitz-Sarkar homotopy type 
to any other $n$-manifold than the 3-sphere, respectively.

Some partial solutions in the Jones polynomial and 
Khovanov homology cases are given in 
 \cite{APS, 
Bourgoin, 
Drobotukhina, 
DKK, 
Kauffman1,Kauffman, Kauffmani, 
Man,  
Igor, 
Ru, 
Tub, 
Viro}: No result has obtained except for the thickened surface or $\R P^3$ case. 

Note that neither \cite {W} or  
 \cite[Theorem 3.3.3, page 560]{RT} gave a partial answer.
\\

The above result \cite{KauffmanNikonovOgasa} is the first meaningful partial solution for KLS homotopy type. This paper is the second one.
\\

\subsection{Homotopical Khovanov homology}\label{subsecHKH}\bb
In this subsection we review {\it $\Z$-homotopical Khovanov homology}  
(\cite{KauffmanNikonovOgasa, MN}).

\begin{definition}\label{def:resolution_configuration}
Let $F$ be a closed oriented surface.
A {\it resolution configuration} $D$, 
or a {\it Kauffman state},  
on the surface $F$ is a pair $(Z(D), A(D))$,
where $Z(D)$ is a set of pairwise-disjoint embedded circles in $F$,
and $A(D)$ is an ordered collection of disjoint arcs embedded in $F$,
with $A(D)\cap Z(D)=\partial A(D)$.

The number of arcs in $A(D)$ is the {\it index} of the resolution configuration $D$,
denoted by $\ind(D)$.

A {\it labeled resolution configuration}, 
or an {\it enhanced Kauffman state}, 
 is a pair $(D, x)$ of a resolution configuration
$D$ and a {\it labeling} $x$ of each element of $Z(D)$ by either $x_+$ or $x_-$.

\end{definition}

\begin{example}
Consider a link $\mathcal L$ in the thickening $F\x[-1,1]$ of $F$. Let $L\subset F$ be a diagram of the link $\mathcal L$.
Assume that the diagram $L$ has $n$ crossings ordered somehow.

For any vector $v\in\{0,1\}^n$ one can define the {\it associated resolution configuration} $D_L(v)$
obtained by taking the resolution of the diagram $L$ corresponding to $v$
(that is, taking the 0-resolution at the $i$-th crossing
if $v_i =0$, and the 1-resolution otherwise)
and then placing arcs corresponding to each of the crossings labeled
by 0's in $v$ (that is, at the $i$-th crossing if $v_i=0$), see Fig.~\ref{fig:resolution}.
For $v=(v_1,...,v_n)$, define $|v|$ to be $(v_1)^2+...+(v_n)^2=v_1+...+v_n$, 
which is called Manhattan norm in \cite{LSk}. 
\label{pageMan}

The index of the associated configuration is $\ind(D_L(v))=n-|v|$.

Let $\Lambda(L)$ be the set of all labeling with $x_+$ and $x_-$ of the associated resolution configurations of the link diagram $L$. 
Note that the elements of this set are enhanced Kauffman states:  
We sometimes call them  Khovanov basis elements (associated with $L$). 

\begin{figure}
\includegraphics[width=120mm]{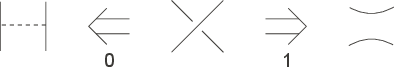}
\caption{\bf The 0- and 1-resolutions}\label{fig:resolution}
\end{figure}
\end{example}

The set $\Lambda(L)$ of Khovanov basis elements has several grading on it.

Let $n$ (respectively, $n_+$, $n_-$) be the number of crossings  (respectively, positive crossings, negative crossings) of $L$.

For a labeled resolution configurations $(D_L(u), x)\in\Lambda(L)$, its
{\em homological grading} is
\begin{equation}\label{eq:homological_grading}
\gr_h(D_{L(u)}, x) = -n_- + |u|,
\end{equation}
and the {\em quantum grading} is

{\normalsize
\begin{equation}\label{eq:quantum_grading}
\gr_q(D_L(u), x) = n_+ - 2n_- + |u|+
\sharp\{Z\in Z(D_L(u)) | x(Z) = x_+\}
-\sharp\{Z\in Z(D_L(u)) | x(Z) = x_-\}.
\end{equation}
}

Let us consider the set $\mathfrak L = [S^1; F]$
of all the homotopy classes of free oriented loops in $F$.  
Let $\bigcirc\in \mathfrak L$ be the homotopy class of contractible loops.
For any closed curve $\gamma$,
one can consider the curve $-\gamma$ obtained from $\gamma$
by the orientation change.
Let $\mathfrak H$ be the quotient group of the free abelian group
with generator set $\mathfrak L$ modulo the relations
 $\bigcirc= 0$ and $[\gamma]=[-\gamma]$ for all free loops $\gamma$.

Define the {\em homotopical grading} of the Khovanov basis element $(D_L(u),x)$ as follows

\begin{equation}\label{eq:homotopical_grading}
\gr_{\mathfrak H}(D_L(u), x)=
\sum_{Z\in Z(D_L(u))} 
\deg x(Z)\cdot[Z] \in\mathfrak H,
\end{equation}
\h
where $\deg(x_\pm)=\pm 1$.

\begin{definition}\label{def:surgery}
  Given a resolution configuration $D$ and a subset $A'\subseteq A(D)$
  there is a new resolution configuration $s_{A'}(D)$, the
  {\em surgery of $D$ along $A'$}, obtained as follows.
  The circles
  $Z(s_{A'}(D))$ of $s_{A'}(D)$ are obtained by performing embedded
  surgery along the arcs in $A'$; in other words, $Z(s_{A'}(D))$ is
  obtained by deleting a neighborhood of $ A'$ from $Z(D)$ and
  then connecting the endpoints of the result using parallel
  translates of $A'$.
  The arcs of $s_{A'}(D)$ are the arcs of $D$ not
  in $A'$, i.e., $A(s_{A'}(D))=A(D)- A'$.

  Let $s(D)=s_{A(D)}(D)$ denote the maximal surgery on $D$.
\end{definition}

\begin{definition}\label{2.10}  
There is a partial order $\prec$ on labeled resolution configurations defined as
follows.
We declare that $(E, y)\prec(D, x)$ if:

\begin{enumerate}
\item
$D$ is obtained from $E$ by surgering along a single arc of $A(E)$

\item
The labelings $x$ and $y$ induce the same labeling on $D\cap E = E\cap D$.

\item
$\gr_q(E, y)=\gr_q(D, x)$, $\gr_{\mathfrak H}(E, y)=\gr_{\mathfrak H}(D, x)$.

\end{enumerate}

The possible cases of the order are drawn in Fig.~\ref{resolC} and~\ref{resolNC}.

\begin{figure}
\includegraphics[width=0.6\textwidth]{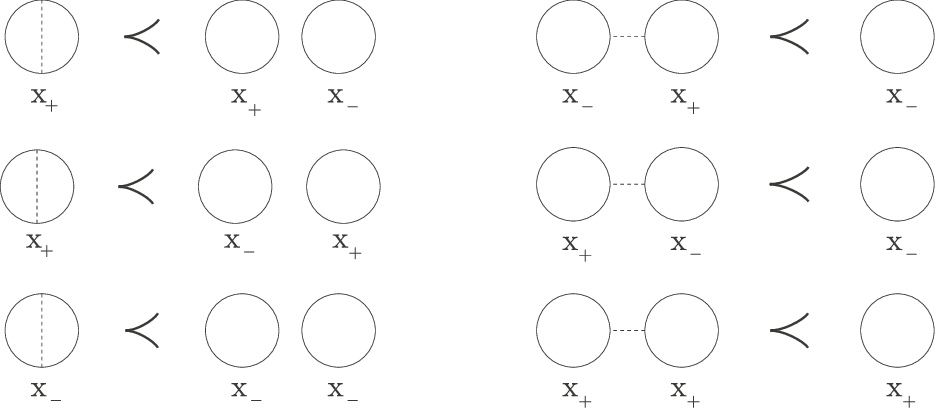}
\caption{{\bf
The partial order of labeled resolution configurations with contractible circles
}\label{resolC}}
\end{figure}

\begin{figure}
\includegraphics[width=0.6\textwidth]{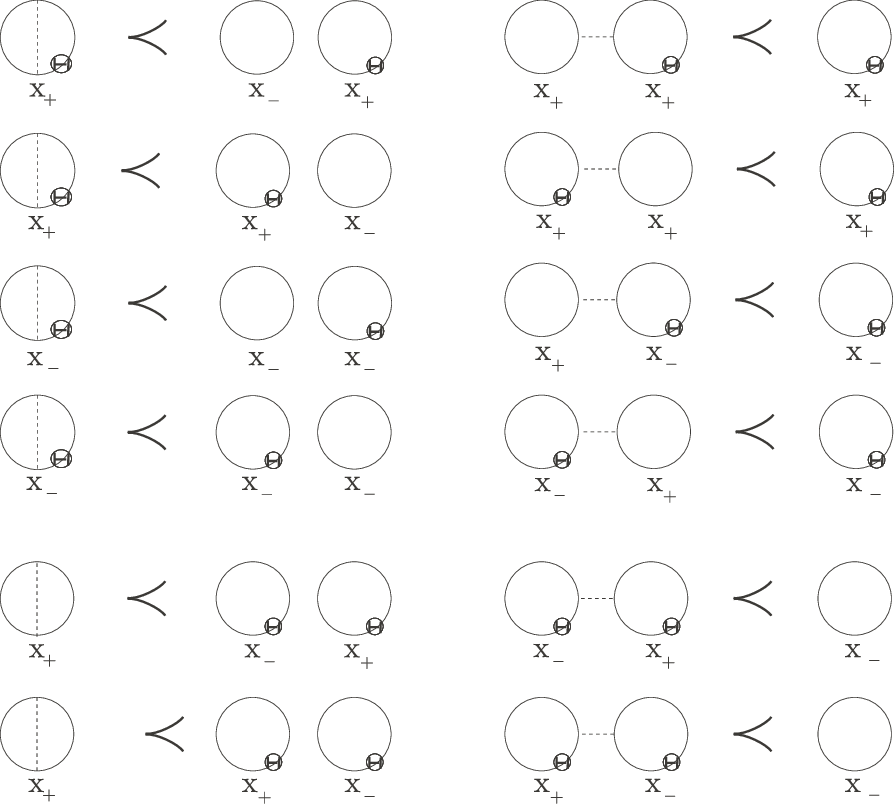}
\caption{{\bf
The partial order of labeled resolution configurations with non-contractible circles.
Non-contractible circles are marked with (H).
}
\label{resolNC}}
\end{figure}

Now, we close the order $\prec$ by transitivity.
\end{definition}

\begin{definition}\label{korekos}
Given an oriented link diagram $L$ with $n$ crossings and an ordering of the crossings in $L$,
the {\em Khovanov chain} complex  $KC(L)$ is defined as the $\Z$-module freely generated
by labeled resolution configurations of the form $(D_L(u), x)$ for $u\in\{0, 1\}^n$.
Thus, the set of all labeled resolution configurations of $L$ is
a basis of $KC(L)$.

The {\em Khovanov differential} preserves the quantum grading and the homotopical grading,
increases the homological grading by 1, and is defined as

\begin{equation}\label{bibun}
{\displaystyle
\delta(D_L(v),y)
=
\sum_{(D_L(u),x)\succ (D_L(v),y)\colon |u|=|v|+1}
(-1)^{s_0(\mathcal C_{u,v})
}}(D_L(u),x),
\end{equation}

\noindent
where for $u= (\epsilon_1,..., \epsilon_{i-1}, 1, \epsilon_{i+1}, . . . , \epsilon_n)$ and
$v=(\epsilon_1,..., \epsilon_{i-1}, 0,\epsilon_{i+1}, . . . , \epsilon_n)$,
one defines $s_0(C_{u,v}) = \epsilon_1+\cdot\cdot\cdot+ \epsilon_{i-1}$.
\end{definition}

\begin{theorem}\label{thmKNO}
$($\cite{KauffmanNikonovOgasa}$)$ 
The homology $KH(L)$ of the complex $(KC(L),\delta)$ are the 
{\em $\Z$-coefficient Khovanov homology} of the link $L$.
\end{theorem}

\begin{definition}\label{def:decorated_resolution_configuration}
A {\em decorated resolution configuration} is a triple $(D, x, y)$
where $D$ is a resolution configuration and
$x$ (respectively,  $y$) is a labeling of each component of $Z(s(D))$
(respectively, $Z(D)$) by an element of $\{x_+, x_-\}$. The labeled resolution configuration $i=(D,y)$ is the {\em  initial configuration} of the decorated resolution configuration, and the labeled resolution configuration $f=(s(D),x)$ is the {\em  final configuration}.

Associated to a decorated resolution configuration $(D, x, y)$ is the poset $P(D, x, y)$
 consisting of all labeled resolution configurations $(E, z)$ with $(D, y)\prec(E, z)\prec(s(D), x)$.
We call $P(D, x, y)$ the poset for $(D, x, y)$.

For any resolution configuration $D'=s_A(D)$, $A\subset A(D)$, we define its {\em multiplicity} to be the number of labelings on $D'$ 
which belong to $P(D,x,y)$:  
\label{pageqqqq}
$$
\mu_{(D, x, y)}(D')=\sharp\{z\,|\, (D, y)\prec(D', z)\prec(s(D), x)\}.
$$
The {\em multiplicity} $\mu(D, x, y)$ of a decorated resolution configuration $(D, x, y)$ is the maximum of the multiplicities $\mu_{(D, x, y)}(D')$:  
$$
\mu(D, x, y)=\max_{A\subset A(D)}\mu_{(D, x, y)}(s_A(D)).
$$
\end{definition}

\begin{definition}\label{def:core_configuration}
The {\it core} $c(D)$ of a resolution configuration $D$ is the resolution configuration
obtained from $D$ by deleting all the circles in $Z(D)$ that are disjoint from all the arcs in $A(D)$.
A resolution configuration $D$ is called {\em basic} if $D = c(D)$, that is, if every circle in $Z(D)$ intersects an arc in $A(D)$.

In the same way one can define the core $c(D, x)$ of a labeled resolution configuration $(D,x)$, and basic labeled resolution configurations.

The core of a decorated resolution configuration $(D,x,y)$ is the decorated configuration
$$c(D,x,y)=(c(D),x\mid_{s(c(D))},y\mid_{c(D)}).$$
A decorated resolution configuration is basic if it coincides with its core.
\end{definition}

\begin{remark}
Given two comparable labeled resolution configurations $\alpha=(D,y)\prec (D',x)=\beta$, $D'=s_A(D)$, $A\subset A(D)$, one can assign a basic decorated resolution configuration ${\mathcal D}(\alpha,\beta)$ to it. Consider two resolution configurations  $\bar D=(Z(D), A)$, $\bar D'=s(\bar D)=(Z(D'),\emptyset)$. Then the decorated configuration ${\mathcal D}(\alpha,\beta)$ is defined as the core of $(\bar D, x, y)$.
\\

If $(D,y)\not\prec (D',x)$ we say that the corresponding decorated resolution configuration is {\it empty}.
\end{remark}

\begin{remark}
(Basic) decorated resolution configurations  form a partially ordered set by inclusion relation: $(D',x',y')\subset (D,x,y)$ if
$(D',y')$ and $(s(D'),x')$ belong the poset $P(D,x,y)$.
\end{remark}

\subsection{Khovanov homotopy type}

Let us remind the construction of Khovanov homotopy type 
for links in $\R^3$, 
which is defined by using diagrams on $\R^2$,  
by R. Lipshitz and S. Sarkar~\cite{LSk}.

\begin{definition}\label{def:manifold_with_corners}
A {\em $k$-dimensional manifold with corners} is a topological space $X$ which is locally homeomorphic to an open subset of $\R^k_+=(\R_+)^k$ where $\R_+=[0,\infty)$.

For $x\in X$, let $c(x)$ be the number of zero coordinates of the corresponding point in $\R^k_+$. The set $\{x\in X\,|\, c(x)=i\}$ is the {\em codimension-$i$ boundary} of $X$.

A {\em connected facet} of $X$ is the closure of a connected component of the codimension-$1$ boundary of $X$. A {\em facet} is a union of disjoint connected facets.
\end{definition}

\begin{definition}\label{def:n_manifold}
A manifold with corners $X$ is called a {\em manifold with facets} if every point $x\in X$ belongs to exactly $c(x)$ connected facets. An $\langle n\rangle$-manifold is a manifold with facets $X$ along with an ordered $n$-tuple $(\partial_1X,\dots,\partial_n X)$ of facets of $X$ such that
\begin{itemize}
\item $\bigcup_{i=1}^n \partial_i X=\partial X$;
\item for all distinct $i,j$ the intersection $\partial_i X\cap \partial_j X$ is a facet of both $\partial_i X$ and $\partial_j X$.
\end{itemize}
\end{definition}
For any $A\subset \{1,\dots,n\}$ denote $X(A)=\bigcap_{i\in A}\partial_i X$.

\begin{definition}\label{def:neat_embedding}
Given a $(n+1)$-tuple ${\mathbf d}=(d_0,\dots,d_n)\in\N^{n+1}$, let
$$
\E^{\mathbf d}_n=\R^{d_0}\times\R_+\times\R^{d_1}\times\R_+\times\dots\times\R_+\times\R^{d_n}.
$$
$\E^{\mathbf d}_n$ is a $\langle n\rangle$-manifold with
$$
\partial_i(\E^{\mathbf d}_n)=\R^{d_0}\times\dots\times\R^{d_{i-1}}\times\{0\}\times\dots\times\R^{d_n}.
$$

A {\em neat immersion} of an $\langle n\rangle$-manifold is a smooth immersion
$\iota\colon X\looparrowright\E_n^{\mathbf d}$ for some $\mathbf d$ such that:
\begin{enumerate}
\item $\iota^{-1}(\partial_i(\E^{\mathbf d}_n))=\partial_i X$ for all $i$,\\
\item for any $A\subset B\subset\{1,\dots,n\}$ the sets $\iota(X(A))$ and $\E^{\mathbf d}_n(B)$ are transversal.
\end{enumerate}

A {\em neat embeding} is a neat immersion that is also an embedding.
\end{definition}

\begin{definition}\label{def:flow_category}
A {\em flow category} is a pair $(\mathscr C, \gr)$ where $\mathscr C$ is a category with finitely many objects $Ob(\mathscr C)$ and $\gr\colon Ob(\mathscr C)\to\Z$ is a function, satisfying the following conditions:

\begin{enumerate}
\item $\Hom(x,x)={id}$ for all $x\in Ob(\mathscr C)$, and for distinct $x,y\in Ob(\mathscr C)$, $\Hom(x,y)$ is a compact $(\gr(x)-\gr(y)-1)$-dimensional $\langle \gr(x)-\gr(y)-1\rangle$-manifold;
\item for distinct $x,y,z\in Ob(\mathscr C)$ with $\gr(z)-\gr(y)=m$ the composition map
$$
\circ\colon \Hom(z,y)\times\Hom(x,z)\to\Hom(x,y)
$$
    is an embedding into $\partial_m\Hom(x,y)$. Furthermore,
$$
\circ^{-1}(\partial_i\Hom(x,y))=\left\{\begin{array}{cl}
\partial_i\Hom(z,y)\times\Hom(x,z) & \mbox{for } i<m,\\
\Hom(z,y)\times\partial_{i-m}\Hom(x,z) & \mbox{for } i>m.\end{array}\right.
$$
\item for distinct $x,y\in Ob(\mathscr C)$ the composition induces a diffeomorphism
$$
 \partial_i\Hom(x,y)\cong\bigsqcup_{z\in Ob(\mathscr C)\colon \gr(z)=\gr(y)+i}\Hom(z,y)\times\Hom(x,z).
$$
\end{enumerate}
\end{definition}

For any objects  $x,y$ in a flow category define the {\em moduli space} from $x$ to $y$ to be
$$
\MM(x,y)=\left\{\begin{array}{cl}
\emptyset & \mbox{if } x=y,\\
\Hom(x,y) & \mbox{otherwise}.\end{array}\right.
$$

Let ${\mathbf d}=(\dots, d_{-1},d_0,d_1,\dots)$ be a sequence of natural numbers. For any $a<b$ denote $\E_{\mathbf d}[a:b]=\E_{b-a-1}^{d_a,\dots,d_{b-1}}$.

\begin{definition}
A neat immersion (embedding) of a flow category $\mathscr C$ is a collection of neat immersions (embeddings) $\iota_{x,y}\colon \MM(x,y)\looparrowright
\E_{\mathbf d}[\gr(y):\gr(x)]$ such that
\begin{enumerate}
\item for all $i,j$ the map
$$\iota_{i,j}=\sqcup_{x,y}\iota_{x,y}\colon
\bigsqcup_{x,y\in Ob(\mathscr C)\colon \gr(x)=i,\gr(y)=j}\MM(x,y)\to \E_{\mathbf d}[j:i]
$$
is a neat immersion (embedding);
\item for all objects $x,y,z$ and all points $p\in\MM(x,z)$, $q\in\MM(z,y)$
$$
\iota_{x,y}(q\circ p)=(\iota_{z,y}(q),0,\iota_{x,z}(p)).
$$
\end{enumerate}
\end{definition}

\begin{definition}\label{def:framed_flow_category}
Let $\iota$ be a neat immersion of a flow category $\mathscr C$. For objects $x,y$, let $\nu_{x,y}$ be the normal bundle on the moduli space $\MM(x,y)$, induced by the immersion $\iota_{x,y}$. A {\em coherent framing} $\phi$ of the normal bundle is a framing for $\nu_{x,y}$ for all objects $x,y$ such that the product framing $\nu_{z,y}\times\nu_{x,z}$ equals to the pullback $\circ^*(\nu_{x,y})$ for all $x,y,z$.

A flow category with a fixed coherent framing of the normal bundle to some neat immersion is called a {\em framed flow category}.
\end{definition}

For a framed flow category there is an associated cochain complex $C^*(\mathscr C)$. The chain space of the complex is the free abelian group generated by the objects of the category: $C^*(\mathscr C)=\Z[Ob(\mathscr C)]$; the differential is given by the formula
$$
\delta y=\sum_{x\in Ob(\mathscr C)\colon \gr(x)=\gr(y)+1} \left(\sum_{f\in\MM(x,y)}\phi(f)\right)x.
$$
The moduli space in the formula is a compact zero-dimensional manifold, i.e. a finite set, and the framing $\phi$ is given by signs of the elements of that set.

To a framed flow category one can associate a based CW complex in the following way.

\begin{definition}\label{def:flow_category_realization}
Let $\mathscr C$ be a framed flow category with a neat embedding $\iota$ into $\E_{\mathbf d}$ and a framing $\phi$.
Let $B=\min_{x\in Ob(\mathscr C)} \gr(x)$ and $B=\max_{x\in Ob(\mathscr C)} \gr(x)$.. Using framing, extend the embedding $\iota_{x,y}$ for some small $\epsilon>0$ to an embedding
$$
\tilde\iota_{x,y}\colon \MM(x,y)\times [-\epsilon,\epsilon]^{d_{\gr(y)}+\cdots+d_{\gr(x)-1}}\to \E_{\mathbf d}[\gr(y):\gr(x)]
$$
where $\E_{\mathbf d}[\gr(y):\gr(x)]=
\R^{d_{\gr(y)}}\times\R_+\times\dots\times\R_+\times\R^{d_{\gr(x)-1}}.$

Choose $R$ sufficiently large so that for all $x,y$
$$\tilde\iota_{x,y}\colon \MM(x,y)\times [-\epsilon,\epsilon]^{d_{\gr(y)}+\cdots+d_{\gr(x)-1}}$$
lies in $[-R,R]^{d_{\gr(y)}}\times[0,R]\times\dots\times[0,R]\times[-R,R]^{d_{\gr(x)-1}}$.

To any object $x$ assign the cell
\begin{multline*}
C(x)=[0,R]\times[-R,R]^{d_{B}}\times\dots\times[0,R]\times[-R,R]^{d_{\gr(x)-1}}\times\{0\}\times\\
[-\epsilon,\epsilon]^{d_{\gr(x)}}\times\dots\times\{0\}\times[-R,R]^{d_{A-1}}.
\end{multline*}
For any other object $y$ such that $\gr(y)<\gr(x)$ one identifies  $C(y)\times\MM(x,y)$ with the subset
\begin{multline*}
C_y(x)=[0,R]\times[-R,R]^{d_{B}}\times\dots\times[0,R]\times[-R,R]^{d_{\gr(y)-1}}\times\{0\}\times\\
\iota_{x,y}(\MM(x,y))\times\{0\}\times[-\epsilon,\epsilon]^{d_{\gr(x)}}\times\dots\times\{0\}\times[-\epsilon,\epsilon]^{d_A-1}.
\end{multline*}

Define the attaching map for $C(x)$ as the map which is the projection $C_y(x)\cong C(y)\times\MM(x,y)$ to $C(y)$ on $C_y(x)$ and the map to the basepoint on the  $\partial C(x)\setminus \bigcup_y C_y(x)$.

These gluing maps define a CW complex $|\mathscr C|$ which is the  {\em Cohen--Jones--Segal realization} of the framed flow category $\mathscr C$.
\end{definition}

\begin{theorem}[\cite{LSk}]
The CW complex $|\mathscr C|$ is well defined and its cellular cochain complex is isomorphic to the associated cochain complex $C^*(\mathscr C)$.
\end{theorem}

\begin{example}[Cube flow category]
Let $X=[0,1]^n$ be the $n$-dimensional cube and $f_n(x_1,\dots,x_n)=f(x_1)+\cdots+f(x_n)$, where $f(x)=3x^2-2x^3$, be a Morse function on it.
Define the {\em $n$-dimensional cube flow category} $\mathscr C_C(n)$ as the Morse flow category of the function $f_n$. This means that the objects of $\mathscr C_C(n)$ are the critical points of $f_n$, i.e. the vertices $\{0,1\}^n$ of the cube $X$.

Denote the object $(0,\dots,0)$ by $\bar 0$, and the object $(1,\dots, 1)$ by $\bar 1$.

The grading function is defined as $\gr(u)=|u|=\sum_{i=1}^n u_i$, $u=(u_1,\dots,u_n)\in\{0,1\}^n$. The moduli space $\MM(x,y)$ consists of the lines of the gradient flow which starts at $x$ and ends at $y$. One can identify the moduli space $\MM(x,y)$ with the permutahedron of dimension $\gr(x)-\gr(y)-1$.

The cube flow category can be framed (by induction on the moduli spaces dimension).
\end{example}

\begin{example}[Khovanov flow category]
Let $\mathcal L$ in $S^3$ be a link and $L$ in $S^2$ be its diagram.

The {\it Khovanov flow category} $\mathscr C_K(L)$ has one object
for each 
Khovanov basis element. 
That is, an object of $\mathscr C_K(L)$ is a labeled resolution configuration of
the form $\mathbf{x}=(D_L(u), x)$  with $u\in\{0, 1\}^n$.
The grading on the objects is
the homological grading gr$_h$; 
the quantum grading gr$_q$ is an additional grading on the objects.
We need the orientation of $L$ in order to define these
gradings, but the rest of the construction of $\mathscr C_K(L)$ is independent of the orientation.
Consider objects $\mathbf{x}=(D_L(u), x)$ and
$\mathbf{y}=(D_L(v), y)$ of $\mathscr C_K(L)$.
The space $\mathcal M_{\mathscr C_K(L)}(\mathbf{x},\mathbf{y})$
is defined to be empty unless $y\prec x$
with respect to the partial order from Definition \ref{2.10}.
So,
assume that $y\prec x$.
Let $x|$ denote the restriction of $x$ to
$s(D_L(v)-D_L(u))=D_L(u)-D_L(v)$
and
let $y|$ denote the restriction of $y$ to $D_L(v)-D_L(u)$.
Therefore, $(D_L(v)-D_L(u), x|, y|)$  is
a basic decorated resolution configuration.

In \cite[\S5 and \S6]{LSk} Lipshitz and Sarkar associate to each index $n$ basic decorated resolution configuration
$(D, x, y)$ an $(n-1)$-dimensional $\left<n-1\right>$-manifold $\mathcal M(D, x, y)$
together with the $\mu(D,x,y)$-fold trivial covering
$$\mathcal F :\mathcal  M(D, x, y) \to\mathcal M_{\mathscr C(n)}(\overline1, \overline0).$$

Use it, and define
$$\MM_{\mathscr C_K(L)}(\mathbf{x},\mathbf{y})=\MM(D_L(v)-D_L(u),x|, y|).$$

The framing of the cube flow category can be lifted to the Khovanov category.
\end{example}

\begin{definition}
The Cohen--Jones--Segal realization ${\mathcal X}(L)=|\mathscr C_K(L)|$ of the Khovanov framed flow category
is called the {\em Khovanov-Lipshitz-Sarkar stable homotopy type}
of the link diagram $L$ on $\R^2$.
\end{definition}

\begin{theorem}[\cite{LSk}]\label{thm:stable_khovanov_homology}
Let $L$ be a link diagram $L$ on $\R^2$ for a link $\mathcal L$ in $\R^3$.   
The Khovanov-Lipshitz-Sarkar stable homotopy type ${\mathcal X}(L)$ defined for 
the link diagram $L$ on $\R^2$  
is a link invariant of the link $\mathcal L$ in $\R^3$.
\end{theorem}

\section{Moduli systems}

Let $F$ be a closed oriented surface. We consider links in the thickening of the surface $F$.
In order to define Khovanov homotopy type of such links we need fix a set of moduli spaces for the decorated resolution configurations of the link. This observation leads to the following definition, cf.~\cite[Section 5.1]{LSk}.

\begin{definition}\label{def:moduli_system}
A {\em moduli system} on the surface $F$ is a family of correspondences between the basic decorated resolution configurations of index $n$ in the surface $F$ and $(n-1)$-dimensional $\langle n-1\rangle$-manifolds:
$$\MM\colon (D,x,y) \mapsto \MM(D, x, y)$$

together with $\langle n-1\rangle$-maps
$$
\mathcal F_{(D, x, y)}\colon\MM(D, x, y)\to\MM_{\CC_C(n)}(\bar 1,\bar 0)
$$

These correspondences must obey the following conditions:

 \begin{enumerate}
 \item the moduli space $\MM(D, x, y)$ of a basic decorated resolution configuration $(D, x, y)$ depends only on the isotopy class of the configuration;
 \item $\MM(D, x, y)=\emptyset$ if $P(D,x,y)=\emptyset$; 
 \item for any $(E,z)\in P(D,x,y)$ there are embeddings
 $$ \circ\colon \MM(D\setminus E, z|, y|)\times \MM(E\setminus s(D), x|, z|)\to \MM(D, x, y);$$

 \item the faces of $\MM(D, x, y)$ are determined by
 $$
 \partial_i\MM(D, x, y)=
 \coprod_{(E,z)\in P(D,x,y),\ ind(D\setminus E)=i} \circ(\MM(D\setminus E, z|, y|)\times \MM(E\setminus s(D), x|, z|));
 $$

 \item the composition is compatible with the maps $\mathcal F$: for any $E=D_D(v)$
  \begin{equation*}\label{eq:moduli_space_equivarity}
  \xymatrix{
    \Moduli(D\sm E,\gen{z}|,\gen{y}|)\times
  \Moduli(E\sm s(D),\gen{x}|,\gen{z}|)
  \ar[r]^-\circ\ar[d]_{{\mathcal F}\times{\mathcal F}} &
  \Moduli(D,\gen{x},\gen{y})\ar[dd]^{{\mathcal F}}\\
  \Moduli_{\CubeFlowCat(n-m)}(\vect{1},\vect{0})\times
  \Moduli_{\CubeFlowCat(m)}(\vect{1},\vect{0})
  \ar[d]_{}&
  \\
  \Moduli_{\CubeFlowCat(n)}({v},\vect{0}) \times
  \Moduli_{\CubeFlowCat(n)}(\vect{1},{v})\ar[r]^-\circ &
  \Moduli_{\CubeFlowCat(n)}(\vect{1},\vect{0}).
  }
   \end{equation*}

 \item the map $\mathcal F_{(D, x, y)}$ is the trivial covering.
 \end{enumerate}
\end{definition}
\vs

In 
\S\ref{sect:decorated} and \S\ref{mod}, 
we will present a moduli system for our case.  
That is, we will prove the following theorem.

\begin{theorem}
\label{thm:moduli_system_existence}
For any closed oriented surface $F$ there exists a moduli system.
\end{theorem}

By Theorem 
\ref{thm:moduli_system_existence}, we have 
Proposition \ref{prbigsite} and 
Theorem \ref{thm:moduli_system_invariance} below.
Given a moduli system $\MM=\{\MM(D, x, y)\}$, let $\mathscr C_\MM(L)$ be the Khovanov flow category whose objects are labeled resolution configurations and moduli spaces are
$$\MM((D_L(u),x),(D_L(v),y))=\MM(D_L(v)-D_L(u),x|, y|).$$

\begin{proposition}\label{prbigsite}
There is a structure of framed flow category on $\mathscr C_\MM(L)$.
\end{proposition}

\begin{proof}
  Let $\iota_C$ be a neat embedding of the cube flow category $\CC_C$ with a coherent framing $\phi_C$. Then $\iota_0=\iota_C\circ\mathcal F$ is a neat map into some $\E_{\mathbf d}[a:b]$.


By Lemma 3.16 of~\cite{LSk} there exists a neat embedding $\iota_1$ of the flow category $\mathscr C_\MM(L)$ into $\E_{\mathbf d'}[a:b]$.
  By Lemma 3.17 of~\cite{LSk} there exists a family of neat maps $\tilde\iota_t$ connecting the neat maps  $\iota_0[\mathbf{d}+\mathbf{d'}+1]$ and $\iota_1[\mathbf{d}+\mathbf{d'}+1]$ such that for any $t>0$ the map $\tilde\iota_t$ is a neat embedding. This family of maps admits an explicit formula
{\normalsize{
\begin{gather*}
 \tilde\iota_t\colon \mathscr C_\MM(L) \to \E_{\mathbf{d}+\mathbf{d'}+1}[a:b]=\R^{d_a+d'_a+1}\times\R_+\times\R^{d_{a+1}+d'_{a+1}+1}\times\R_+\times\dots\times\R_+\times\R^{d_{b-1}+d'_{b-1}+1},\\
 \tilde\iota_t=((1-t)\iota_0^{a}+t\iota_1^a,f(t)\iota_1^a,f(t)\bar\iota_1^{a+1},(1-t)\bar\iota_0^{a+1}+t\bar\iota_1^{a+1}, (1-t)\iota_0^{a+1}+t\iota_1^{a+1},\dots,
 (1-t)\iota_0^{b-1}+t\iota_1^{b-1},f(t)\iota_1^{b-1},0)
\end{gather*}
where $f(t)=e^{-\frac 1{t(1-t)}}$ and
$$
\iota_0=(\iota_0^a,\bar\iota_0^{a+1},\iota_0^{a+1},\dots,\bar\iota_0^{b-1},\iota_0^{b-1})\colon  \mathscr C_\MM(L) \to \R^{d_a}\times\R_+\times\R^{d_{a+1}}\times\dots\times\R_+\times\R^{d_{b-1}}.
$$}}

By Lemma~3.19 of~\cite{LSk} we can extend the coherent framing $\phi$ for the map $\tilde\iota_0=\iota_0[\mathbf{d'}+\mathbf{d''}]$ to a family of coherent framings $\phi_t$ for the maps $\tilde\iota_t$. In particular, the neat embedding $\tilde\iota_1=\iota_1[\mathbf{d}+\mathbf{d''}]$ admits the coherent framing $\phi_1$.
\end{proof}\bb

Define the Khovanov homotopy type ${\mathcal X}_\MM (L)$ associated with the moduli system $\MM$ as the realization of the framed flow category $\mathscr C_\MM(L)$.

\begin{theorem}\label{thm:moduli_system_invariance}
 Khovanov homotopy type ${\mathcal X}_\MM (L)$ is a link invariant.
\end{theorem}

\begin{proof}
We can use the reasoning of~\cite[Propositions 6.2, 6.3, 6.4]{LSk} without any changes.

Indeed, let the diagram $L'$ differ from $L$ by an increasing first Reidemeister move, Fig.~\ref{fig:reidemeister1}.

\begin{figure}[h]
\includegraphics[width=0.25\textwidth]{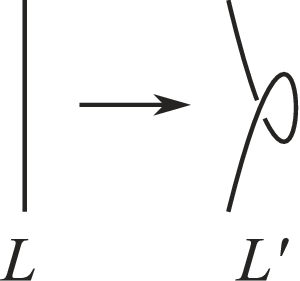}
\caption{{\bf First Reidemeister move
}\label{fig:reidemeister1}}
\end{figure}

The Khovanov complex of the diagram $L'$ contains a contractible subcomplex $C_1$, Fig.~\ref{fig:reidemeister1_complex} left. The quotient complex $C_2$ (Fig.~\ref{fig:reidemeister1_complex} right) can be identified with the Khovanov complex of $L$. On the level of flow category this means that the Khovanov flow category $\mathscr C_K(L)$ contains a closed subcategory $\mathscr C_1$ which corresponds to $C_1$ and a subcategory $\mathscr C_2$ that corresponds to $C_2$. The subcategory $\mathscr C_2$ is isomorphic to the category $\mathscr C_K(L)$. The geometric realization of $\mathscr C_1$ is contractible, hence, ${\mathcal X}_\MM (L')=|\mathscr C_K(L')|=|\mathscr C_1|\vee |\mathscr C_2|=|\mathscr C_2|=|\mathscr C_K(L)|={\mathcal X}_\MM (L)$.

\begin{figure}[h]
\includegraphics[height=0.13\textheight]{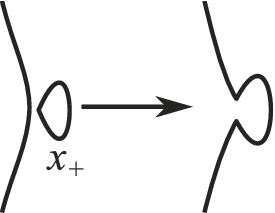}\hfil\qquad
\includegraphics[height=0.13\textheight]{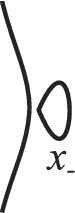}
\caption{{\bf Parts of Khovanov complex of the diagram $L'$
}\label{fig:reidemeister1_complex}}
\end{figure}

Analogously, one establishes the invariance under the second and third Reidemeister moves.
\end{proof}


\section{Decorated resolution configurations}\label{sect:decorated}

Let $F$ be a closed oriented surface. Let us describe decorated resolution configurations of link diagrams in the surface $F$ in more details.

Let $\mathcal D = (D,x,y)$ be a decorated resolution configuration in $F$. 
It is a trivalent graph on $F$ 
\label{pageon} 
consisting of cycles and arcs between the cycles.

Recall that $P(D,x,y)$ is the partially ordered set consisting of the labeled resolution configurations between $(D,y)$ and $(s(D),x)$.

Let $A=A(D)$ be the set of arcs. 
For a finite set, let $|A|$ be the number of elements of $A$. 
Assuming the arcs are ordered, there is a map $\pi\colon P(D,x,y)\to C_n$, where $n=|A|$ is the index of $D$ and $C_n\simeq\{0,1\}^n$ is the vertex set of $n$-dimensional cube. If $(s_{A'}(D),z)\in P(D,x,y)$ where $A'\subset A=\{a_i\}_{i=1,\dots,n}$ then one sets $\pi(s_{A'}(D),z)=\chi_{A}=(\epsilon_1,\dots,\epsilon_n)$ where $\epsilon_i=1$ if $a_i\in A'$ and $\epsilon_i=0$ if not.

Recall that the multiplicity 
of the decorated resolution configuration is the number 
$\mu(D,x,y)=\max_{v\in C_n} |\pi^{-1}(v)|$. \label{page||}
We abbreviate $\mu(D,x,y)$ to $\mu(D)$ when it is clear from the context. 
Let $\mathcal D_i = (D_i, x|_{s(D_i)}, y|_{D_i})$, $i=1,...,l$, \label{pageccd}
be the connected components of the decorated resolution configuration $(D,x,y)$.

\begin{proposition}
$\mu(\mathcal D)=\mu(\mathcal D_1)\times\cdots\times\mu(\mathcal D_l)$.
\end{proposition}

\begin{proof}
Indeed, for each component $\mathcal D_i$ we can take a vector $v_i$ with the maximal preimage $\pi^{-1}|_{\mathcal D_i}(v_i)$. Then the concatenation $v$ of the vectors $v_1,\dots,v_l$ gives the maximal preimage of the map $\pi$, and $|\pi^{-1}(v)|=\prod_{i=1}^l \left|\pi^{-1}|_{\mathcal D_i}(v_i)\right|$.
\end{proof}

Below we assume that the decorated resolution configuration $D$ is connected.

Let $(D,y)$ be a labeled resolution configuration. Denote the number of circles in $D$ by $\gamma(D)$. 
 Let $\sigma(y_\ast)$ be $1$ if $y_\ast=y_+$, and $-1$ if $y_\ast=y_-$.    
 Let $|y|$ be the sum $\sigma(y_1)+...+\sigma(y_{\gamma(D)})$. 
 Then $-\gamma(D)\le|y|\le\gamma(D)$. 
The quantum grading (igonoring some terms) 
of the enhanced state $(D,y)$ is equal to  $|\pi(D)|+|y|=|y|$. \label{pagetoshiki?}   
%
 Let $A'\subset A(D)$ and $(s_{A'}(D),z)\succ (D,y)$. Since the quantum grading (igonoring some terms) 
 of comparable configurations coincide, 
$|y|=|\pi(s_{A'}(D))|+|z|=|A'|+|z|$. 
Hence, $-\gamma(s_{A'}(D))\le|z|=|y|-|A'|$ and $\gamma(s_{A'}(D))\ge |A'|-|y|$.

In particular, if $\gamma(D)=1$ and $|y|=-1$ then $\gamma(s_{A'}(D))\ge |A'|+1$, hence,  $\gamma(s_{A'}(D))=|A'|+1$.
If $\gamma(D)=1$ and $|y|=1$ then $\gamma(s_{A'}(D))=|A'|\pm 1$.

\begin{proposition}
  Let $a\in A(D)$ be a leaf or coleaf $($see Fig. $\ref{fig:leaf_coleaf})$. 
  Then there exists a unique labeled resolution configuration $(s_a(D),z)$ such that $(D,y)\prec(s_a(D),z)\prec(s(D),x)$, and
  $$P(D,x,y)=P(s_a(D),z,y)\times\{0,1\}.$$

\begin{figure}[h]
\includegraphics[width=0.4\textwidth]{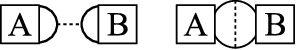}
\caption{{\bf Leaf (left) and coleaf (right) in a resolution configuration
}\label{fig:leaf_coleaf}}
\end{figure}
\end{proposition}

\begin{proof}
If $a$ is a leaf then $(s_a(D),z)$ is uniquely determined. If $a$ is a coleaf then $s_a(D)$ splits into two components $(D_1,z_1)$ and $(D_2,z_2)$. Then $s(D)=s(D_1)\sqcup s(D_2)$. The ambiguity of $z_1, z_2$ concerns the labels of the circles incident to $a$. But one can restore the labels using the quantum grading  (igonoring some terms) 
of the component:  $|\pi(D_i)|+|z_i|=|\pi(s(D_i))|+|x|_{s(D_i)}|$, so  $|z_i|=|\pi(s(D_i))|+|x|_{s(D_i)}|-|\pi(D_i)|$. If one knows the labels of all the circles in $(D_i,z_i)$ except one and the sum of all labels $|z_i|$, then the last label is uniquely determined.

Let $P_0=\{(s_{A}(D),u)\in P(D,x,y)\,|\, a\not\in A\}$  and $P_1=\{(s_{A}(D),u)\in P(D,x,y)\,|\, a\in A\}$. Using the reasonings above, we can prove that the surgery $s_a$ establishes an bijection between $P_0$ and $P_1=P(s_a(D),z,y)$. This bijection is compatible with the order in $P_0$ and $P_1$.
\end{proof}

Let $(D,x,y)$ be a nonempty decorated resolution configuration with one circle. 
Then $D$ is a chord diagram. 
Let $M=(m_{ab})_{a,b\in A(D)}$ be 
the $\Z_2$-valued {\it interlacement matrix}: \label{pageim}
Note $m_{aa}=0$.  

$m_{ab}=lk(a,b)$, that is  $m_{ab}=1$ if the arcs $a$ and $b$ are linked, and $m_{ab}=0$ if they are not. 
Here, `linking', $lk$, means that the boundary of $a$ and that of $b$ are linked in the loop $D$.  
\label{pageIgor1}   

Then $M$ is a symmetric matrix with zero diagonal.

For any labeled resolution $(s_{A'}(D),z)\in P(D,x,y)$ the number of circles $\gamma(s_{A'}(D))$ is given by the circuit-nullity formula. \label{pageuori}

\begin{proposition}\label{pr|}
  $\gamma(s_{A'}(D))=corank M|_{A'}+1$.
\end{proposition}

On the other hand, the quantum grading (igonoring some terms) gives the equality  $|z|+|A'|=|y|$, so $\gamma(s_{A'}(D))\ge |A'|-|y|$.

If $|y|=-1$ then $\gamma(s_{A'}(D))=|A'|+1$ and all labels in $z$ are equal to $x_-$. Thus, $\mu(D,x,y)=1$.

If $|y|=1$ then $\gamma(s_{A'}(D))\ge |A'|-1$. Hence,
$$rank M|_{A'}=|A'|-corank M|_{A'}=|A'|-\gamma(s_{A'}(D))+1\le 2.$$
In particular, $rank M\le 2$. 
Since $M$ is symmetric with zero diagonal, then $rank M$ is even.

If $rank M=0$ then $M=0$ and all the arcs are coleaves. Thus, $\mu(D,x,y)=1$.

If $rank M=2$ then there are two independent vectors $a,b\in\Z_2^n$ such that any row of the matrix is equal to $0,a,b$ or $a+b$. This means the arcs of the chord diagram split into four subsets: three sets of parallel chords and coleaves, see Fig.~\ref{fig:one_circle_state}. Then $A(D)=A_a\sqcup A_b\sqcup A_{a+b}\sqcup A_0$.

\begin{figure}[h]
\includegraphics[width=0.15\textwidth]{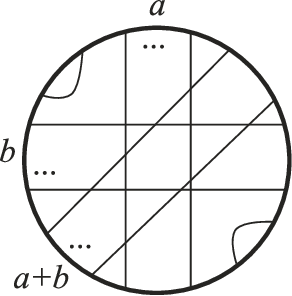}
\caption{{\bf Combinatorial structure of a non-empty decorated resolution configuration with one circle}\label{fig:one_circle_state}}
\end{figure}

\begin{proposition}\label{prop:decorated_resolution_1}
  Let $(D,x,y)$ be a non-empty decorated resolution configuration with one circle such that its interlacement matrix has $rank M = 2$. Then
  
  \begin{enumerate}
    \item[$(1)$] the circle of the configuration $D$ is contractible
    \item[$(2)$] the arcs which belong to one subset $A_a, A_b, A_{a+b}$ are homotopical in $F$.
  \end{enumerate}
\end{proposition}

\begin{proof}
(1) Assume $D$ is not contractible. Take arcs $a\in A_a$, $b\in A_b$. Let $D'=s_{\{a,b\}}(D)$. Then $D$ has label $x_+$ and $D'$ has label $x_-$, hence, $\gr_{\mathfrak H}(D,x_+)=[D]$ and $\gr_{\mathfrak H}(D',x_-)=-[D']$. But $[D]\ne -[D']$.

(2) Let $a_1,a_2\in A_a$ be two non-homotopical arcs, and $b\in A_b$. Then the resolution $s_{\{a_1,a_2, b\}}(D)$ consists of two circles of homotopy type $[a_1b^\pm a_2^{-1}b^\mp]$ and $[a_1^{-1}a_2]$ with labels $x_-,x_-$. Then the homotopical grading of this labeled resolution can not be zero. But $\gr_{\mathfrak H}(D,x_+)=0$ because $D$ is contractible.
 \end{proof}

Let $(D,x,y)$ be as 
in Proposition 
\ref{prop:decorated_resolution_1}. 
The complement to the circle of the resolution configuration consists of two components, one of which is contractible. We will call an arc {\em inner} if it lies in the contractible component, and {\em outer} otherwise.

\begin{proposition}\label{prop:decorated_resolution_2}
  Let $(D,x,y)$ be as 
in Proposition 
\ref{prop:decorated_resolution_1}. 
Then
  \begin{enumerate}
    \item[$(1)$] the arcs from one of the subsets $A_a, A_b, A_{a+b}$ are either all inner or all outer
    \item[$(2)$] let the subsets $A_a, A_b$ consist of outer arcs. Then the surface $F$ is the torus, and the homology type of the arcs  in $A_{a+b}$ 
    $($when $A_{a+b}\ne\emptyset)$ is the sum of homology classes of an arc in $A_a$ and an arc in $A_b$.
    \item[$(3)$] let the subset $A_a$ consist of inner arcs. Then $A_b$ consists of outer arcs and $A_{a+b}=\emptyset$.
  \end{enumerate}
\end{proposition}

\begin{proof}
  (1) Assume that $a_1,a_2\in A_a$ and $b\in A_b$. Since $a_i$, $i=1,2$, and $b$ are interlaced and do not intersect, then these chords lie in different components to the circle of $D$. Hence, the chords $a_1$ and $a_2$ lie in one component, i.e. they are both internal or both external. Thus, the chords from the subset $A_a$ are all internal or all external. The same statements holds for the subsets $A_b$ and $A_{a+b}$.

  (2) Let $a\in A_a$ and $b\in A_b$ be outer chords. The intersection index of the loops $a$ and $b$ is equal to $1$, hence, $a$ and $b$ present independent homology classes. The resolution configuration $s_{\{a,b\}}(D)$ consists of one circle with the label $x_-$. This circle must be contractible due to the homotopical grading. Hence, the surface $F$ can be obtained by gluing two disc along the arcs $a$ and $b$. Thus, $F$ is the torus. Let $c\in A_{a+b}$. Since $c$ is disjoint from $a$ and $b$ and $F$ is the torus, the homology class corresponding to $c$ is equal to $a\pm b$, i.e. it is the sum of the homology classes of $a$ and $b$ (up to signs).

  (3)  
  Let $a\in A_a$ be inner and $b\in A_b$ and $c\in A_{a+b}$. Then $b$ and $c$ are outer chords. The intersection index of the loops $b$ and $c$ is equal to $1$, hence, $b$ and $c$ present independent homology classes. The resolution configuration $s_{\{a,b,c\}}(D)$ consists of two circles with the labels $x_-$. Then the homotopical grading of this labeled resolution is equal $-2[bc^{-1}]\ne 0$ but the initial resolution configuration has homotopical grading $0$. Thus, the decorated resolution configuration must be empty.
\end{proof}

Note that the necessary conditions of 
Propositions \ref{prop:decorated_resolution_1} and  
%
%
%
 \ref{prop:decorated_resolution_2} 
are also sufficient for the decorated resolution configuration to be 
non-empty.

\begin{proposition}\label{prop:1circle_multiplicity}   
  Let $(D,x,y)$ be as 
 be as in Proposition 
\ref{prop:decorated_resolution_1}. 

$(1)$ The multiplicity of a resolution configuration $s_{A'}(D)$ is equal to $2$ if $A'$ contains chords from exactly one of subsets $A_a, A_b, A_{a+b}$, and is equal to $1$ otherwise.

$(2)$ Let $(D',z,w)\subset (D,x,y)$ be a decorated resolution configuration with an initial configuration $(s_{A'}(D),w)$ and final configuration $(s(A''),z)$, where $D'=s_{A'}(D)$ and $A'\subset A''\subset A(D)$. Then $\mu(D',z,w)=2$ if and only if $A'\subset A_0$ and $A''$ intersects at least two of the subsets $A_a$, $A_b$, $A_{a+b}$.
\end{proposition}

\begin{proof}
  (1) Let $A'\subset A(D)$. If $A'$ does not intersect with $A_a, A_b, A_{a+b}$ then it contains only coleaf arcs. The resolution configuration $s_{A'}(D)$ consists of circles among which only one contains a pair of linked arcs. Then the label of this circle in $s_{A'}(D)$ must be $x_+$ (otherwise the surgery by the pair of linked arcs gives zero) whereas the labels of the other circles must be $x_-$. Thus, the labels of the configuration $s_{A'}(D)$ are defined uniquely, and $\mu(s_{A'}(D))=1$.

  If $A'$ contains a pair arcs from different subsets $A_a, A_b, A_{a+b}$ then $rank M|_{A'}=2$. This means that the labels of all the circles in $s_{A'}(D)$ are $x_-$. Thus, $\mu(s_{A'}(D))=1$.

  If $A'$ contains chords from exactly one of subsets $A_a, A_b, A_{a+b}$ then the resolution configuration looks like in Fig.~\ref{fig:resolution_mult2}. There are two circles connected with an arc in $s_{A'}(D)$. Hence, the label of one of these circle must be $x_+$. The labels of the other circles are $x_-$. Hence, the multiplicity is $2$.

\begin{figure}[h]
\includegraphics[width=0.4\textwidth]{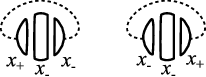}
\caption{{\bf Labelings of a resolution configuration of multiplicity $2$}\label{fig:resolution_mult2}}
\end{figure}

 (2) The condition of the second statement means the decorated configuration $(D',z,w)$ contains a resolution configuration of multiplicity $2$. Thus, it follows from the first statement of the proposition.

\end{proof}

Let us now consider resolution configurations with several circles in the initial state.

Let $(D,x,y)$ be a connected nonempty decorated resolution configuration with $k\ge 1$ circles in the initial state. Then there are $k-1$ arcs $\hat A=\{c_1,\dots,c_{k-1}\}\in A(D)$ such that the resolution $D'=s_{\hat A}(D)$ consists of one circle (Fig.~\ref{fig:decorated_many_circles}). In the diagram $D'$ we can draw the arc adjoint to the arcs of $\hat A$. With some abuse of notation, we denote the set of the adjoint arcs by $\hat A$. Then $D=s_{\hat A}(D')\cup\hat A$.

\begin{figure}[h]
\includegraphics[width=0.45\textwidth]{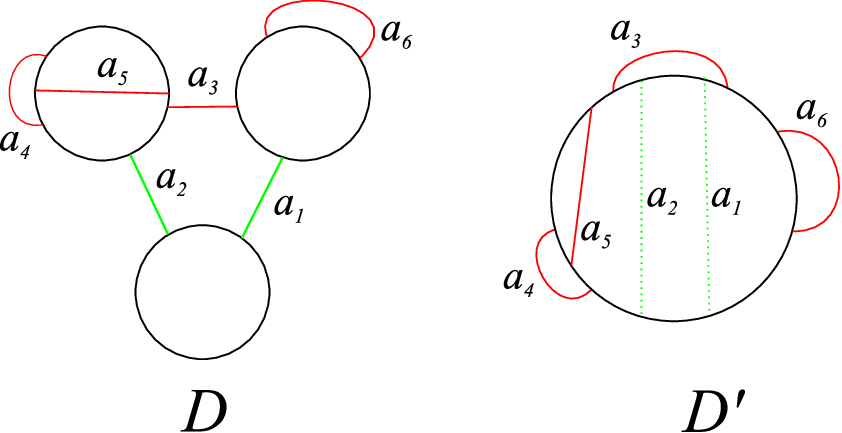}
\caption{Decorated resolution configurations $D$ and $D'$. The set $\hat A$ consists of green arcs}\label{fig:decorated_many_circles}
\end{figure}

Let $M$ be the interlacement matrix of  the chords $A(D')\cup \hat A\simeq(A(D)\setminus\hat A)\cup\hat A=A(D)$ on the circle $D'$. 
Then for any $A'\subset A(D)$ 
the number of circles \label{pageof}    
in the resolution configuration is equal to $corank M|_{A'\bigtriangleup\hat A}+1$ where $A'\bigtriangleup\hat A$ is the symmetric difference of the sets $A'$ and $\hat A$.

\begin{proposition}\label{prop:kcircle_multiplicity}
Let $(D,x,y)$ be a connected nonempty decorated resolution configuration such that $rank M|_{A(D)\setminus\hat A}=2$ and $A'\subset A(D)$. Let $M_{A'}$ be the submatrix of the interlacement matrix $M$ whose rows correspond to the subset $A'\setminus\hat A$ and columns correspond to the subset $A(D)\setminus(A'\cap \hat A)$, and $M^0_{A'}$ be the submatrix with rows from $A'\setminus\hat A$ and columns from $\hat A\setminus A'$. Then the resolution configuration $s_{A'}(D)$ has in $(D,x,y)$ multiplicity $2$ if and only if $rank M_{A'}-rank M^0_{A'}=1$, and multiplicity $1$ otherwise.
\end{proposition}

\h{\bf Remark.}  \label{pagechiga}  
$M|_{A'}$, $M_{A'}$ and $M_{A'}^0$ are different notations.

\bb
\begin{proof}
  The matrix $M_{A'}$ includes $M_{A'}^0$, so $rk M_{A'}-rk M^0_{A'}\ge 0$.

  Let $\tilde D=s_{\hat A\cap A'}(D)$ be the resolution configuration obtained from $D$ by surgery along arcs in $\hat A$. The resolution configuration $\tilde D$ has $|\hat A\setminus A'|+1$ circles with labels $x_+$.

  Assume that $rk M^0_{A'}=0$. This means that the resolution configuration $\tilde D$ consists of $|\hat A\setminus A'|+1$ distinct components.

  The chords from different components are not interlaced, i.e. the coefficient correspondent to them in the matrix $M_{A'}$ is zero. Indeed, let $c$ and $c'$ belong to different components. Then the resolution configuration $s_{\{c,c'\}}(\tilde D)$ has $|\hat A\setminus A'|+3$ circles, and the resolution configuration $s_{\{c,c'\}\cup (\hat A\setminus A')}(\tilde D)$ has $|\hat A\setminus A'|+3-|\hat A\setminus A'|=3$ circles. But the number of circles is equal to $corank M(\{c,c'\})+1$. Hence, $corank M(\{c,c'\})=2$ and $M(\{c,c'\})=0$, so $m_{cc'}=0$.

  Among the components of $\tilde D$ only one can include interlaced chords. Indeed, if we have pairs $c_1,c_1'$ and $c_2,c_2'$ of interlaced chords from different component then the surgery $s_{\{c_1,c_1',c_2,c_2'\}\cup (\hat A\setminus A')}(\tilde D)$ would give a resolution configuration with one circle and a label of degree $-2$, but $\deg(x_-)=-1$ is the minimal grading. Thus, the label is zero and the decorated resolution configuration $(D,x,y)$ is empty.

  Thus, we have only one component with interlaced chord. For the matrix $M(A')$ this means that its nonzero elements can correspond only to the distinguished component. Thus, $rk M_{A'}$ is equal to the rank of the submatrix which corresponds to the component with interlaced chords. Then the statement of the proposition follows from Proposition~\ref{prop:1circle_multiplicity}.

  If $rk M^0_{A'}>0$ then there are chords in $\tilde D$ that connect different circles. Then we can chose subsets $A''\subset A'$ and $\hat A''\subset \hat A\setminus A'$ such that $|A''|=|\hat A''|=rank M^0_{A'}$ and surgery along the set $\hat A\setminus (A'\cup \hat A'')\cup A''$ transforms the resolution configuration $\tilde D$ to a configuration with one circle. Let $\hat A_1=(\hat A\setminus\hat A'')\cup A''$ and $M'$ be the interlacement matrix on the circle $D_1=s_{\hat A_1}(D)$. Then one can check that $corank M|_{A'\bigtriangleup\hat A}=corank M'|_{A'\bigtriangleup\hat A_1}$,  $rank M_{A'}-rank M^0_{A'}=rank M'_{A'}-rank (M')^0_{A'}$ and $rank (M')^0_{A'}=0$. Thus, the statement of the proposition reduces to the case considered above.

\end{proof}

\begin{example}\label{exDyaro}  
Consider the resolution configuration $D$ in Fig. \ref{fig:decorated_many_circles}. 
With labels $x_+$ on each circle, it defines an initial labeled resolution configuration of a decorated resolution configuration. One can choose the set $\hat A=\{a_1,a_2\}$ to merge the circles of the configuration. The interlacement matrix then is equal to
$$
M=\left(\begin{array}{cccccc}
0 & 0 & 1 & 0 & 0 & 0\\
0 & 0 & 1 & 0 & 0 & 0\\
1 & 1 & 0 & 0 & 0 & 0\\
0 & 0 & 0 & 0 & 1 & 0\\
0 & 0 & 0 & 1 & 0 & 0\\
0 & 0 & 0 & 0 & 0 & 0
\end{array}\right).
$$

Consider the surgery along the subset $A'=\{a_1,a_3,a_5,a_6\}$. 

$$
M_{A'}=\left(\begin{array}{c|ccc}
1 & 0 & 0 & 0 \\
0 & 0 & 1 & 0 \\
0 & 0 & 0 & 0
\end{array}\right)
$$
where the left part is the submatrix $M^0_{A'}$. Then $rank M_{A'}=2$ and $rank M^0_{A'}=1$. Since $rank M_{A'}-rank M^0_{A'}=1$, the multiplicity of the resolution configuration $s_{A'}(D)$ is equal to $2$. The labelings of the resolution configuration $s_{A'}(D)$ are shown in Fig.~\ref{fig:decorated_many_circles_surgery}.

\begin{figure}[h]
\includegraphics[width=0.7\textwidth]{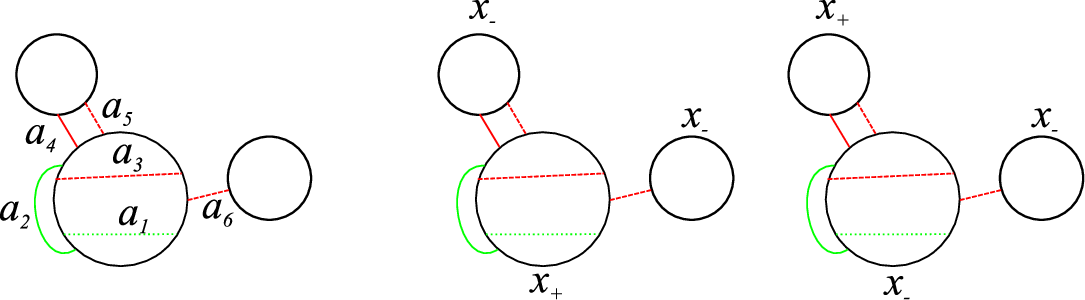}
\caption{Resolution configuration $s_{\{a_1,a_3,a_5,a_6\}}(D)$ and its labelings}\label{fig:decorated_many_circles_surgery}
\end{figure}

\end{example}

Let us enumerate the isotopy classes of connected decorated resolution configuration of multiplicity $>1$.

\subsection{Decorated resolution configurations of index $1$}

All nonempty decorated resolution configurations consist of a comparable pair of labeled resolution configurations and have multiplicity $1$.

\subsection{Decorated resolution configurations of index $2$}

According to Propositions~\ref{prop:1circle_multiplicity} and~\ref{prop:kcircle_multiplicity} the multiplicity of a decorated resolution configuration $(D, x, y)$ with two arcs is equal to $1$ except three cases (see Fig.~\ref{fig:ladybug_cases}).


If $F=S^2$ then there is a unique up to isomorphism decorated resolution configuration $L_0$ of multiplicity $2$. 
When $F\ne T^2, S^2$, 
the decorated resolution configurations of multiplicity $2$ are the ladybug resolution configuration of type $L_\alpha$ where $\alpha$ is a homotopy class of a nontrivial simple loop in $F$. In the torus case $F=T^2$, there is a two-parameter series of decorated resolution configurations of multiplicity $2$. The parameters are two homology classes $\alpha$ and $\beta$ in $H_1(T^2)$.

    \begin{figure}[h]
\centering\includegraphics[width=0.6\textwidth]{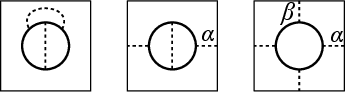}
\caption{Ladybug configurations of type $L_0$ (left) and $L_\alpha$ (middle), and a quasi-ladybug configuration of type $Q_{\alpha,\beta}$ (right). The labels of the circles are $x_+$}\label{fig:ladybug_cases}
\end{figure}

Let us consider these configuration more attentively.

\subsubsection{\bf The ladybug configuration for link diagrams in $S^2$}\label{tamago}

\noindent
We review the ladybug configuration for link diagrams  in $S^2$,
which is introduced in \cite[section 5.4]{LSk}.
Lipshitz and Sarkar introduced it in the case of link diagrams in $S^2$.
We cite the definition of it,
that of the right pair, and that of the left pair associated with it
from \cite[section 5.4.2]{LSk}. \\

\begin{definition}\label{teten}  {\bf (\cite[Definition 5.6]{LSk}).}
An index 2 basic resolution configuration $D$ in $S^2$
is said to be a ladybug configuration
if the following conditions are satisfied (See Figure \ref{tento}.).

$\bullet$ $Z(D)$ consists of a single circle, which we will abbreviate as $Z$;

$\bullet$ The endpoints of the two arcs in $A(D)$, say $A_1$ and $A_2$,
alternate around $Z$

\hskip3mm (that is, $\partial A_1$ and $\partial A_2$ are linked in $Z$).\\

\end{definition}

\begin{definition}\label{rl}  {\bf (\cite[section 5.4.2]{LSk}).}
Let $D$ be as above.
Let $Z$ denote the unique circle in $Z(D)$.
The surgery $s_{A_1}(D)$ (respectively,  $s_{A_2}(D)$) consists of two circles;
denote these $Z_{1,1}$ and $Z_{1,2}$ (respectively, $Z_{2,1}$ and $Z_{2,2}$);
that is, $Z(s_{A_i}(D)) = \{Z_{i,1}, Z_{i,2}\}$.

As an intermediate step, we distinguish two of the four arcs in
$Z - (\partial A_1\cup \partial A_2)$.
Assume that the point $\infty\in S^2$ is not in $D$,
and view $D$ as lying in the plane $S^2-\{\infty\}\cong\R^2$.
Then one of $A_1$ or $A_2$ lies outside $Z$ (in the plane)
while the other lies inside $Z$.
Let $A_i$ be the inside arc and $A_o$ the outside arc.
The circle $Z$ inherits an orientation from the disk it bounds in $\R^2$.
With respect to this orientation, each component of
$Z - (\partial A_1\cup\partial A_2)$
either runs from the outside arc $A_o$ to an inside arc $A_i$ or vice-versa.
The {\it right pair} is the pair of components of $Z-(\partial A_1\cup\partial A_2)$
which run from the outside arc $A_o$ to the inside arc $A_i$.
The other pair of components is the {\it left pair}. See \cite[Figure 5.1]{LSk}.
\end{definition}

We explain why the ladybug configuration is important, below.

\begin{proposition}\label{4}
Let ${\bf x}$ $($respectively, ${\bf y})$ be a labeled resolution configuration in $S^2$
of
homological grading $n$ $($respectively, $n+2).$
 Then the cardinality of the set

\hskip3mm $\{p| p$ is a labeled resolution configuration.
${\bf x}\prec p, p\prec{\bf y},$ $p\neq{\bf x}$, $p\neq{\bf y}\}$

\noindent
is 0, 2, or 4, where $\prec$ represents the partial order defined in
\cite[Definition 2.10]{LSk}.
  \end{proposition}

Let $D$ be the ladybug configuration  in $S^2$.
Since each of $D$ and $s(D)$ has only one circle,
we can let $x_+$ or $x_-$ denote a labeling on it.
Give $D$ (respectively, $s(D)$) a labeling $x_+$ (respectively, $x_-$).
We call the resultant labeled resolution configuration
 $(D,x_+)$ (respectively, $(s(D),x_-)$).
We obtain a decorated resolution configuration $(D,x_-,x_+)$
as drawn in Figure \ref{uenp}.

\begin{fact}\label{tenten}
The case of 4 in Proposition \ref{4} occurs in the above case $(D,x_-,x_+)$.
\end{fact}

Fact \ref{tenten}  is also explained in \cite[section 5.4]{LSk}.

\begin{figure}
\includegraphics[width=0.16\textwidth]{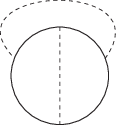}
\caption{{\bf The ladybug configuration
}\label{tento}}
\end{figure}

\begin{figure}
\centering\includegraphics[width=0.65\textwidth]{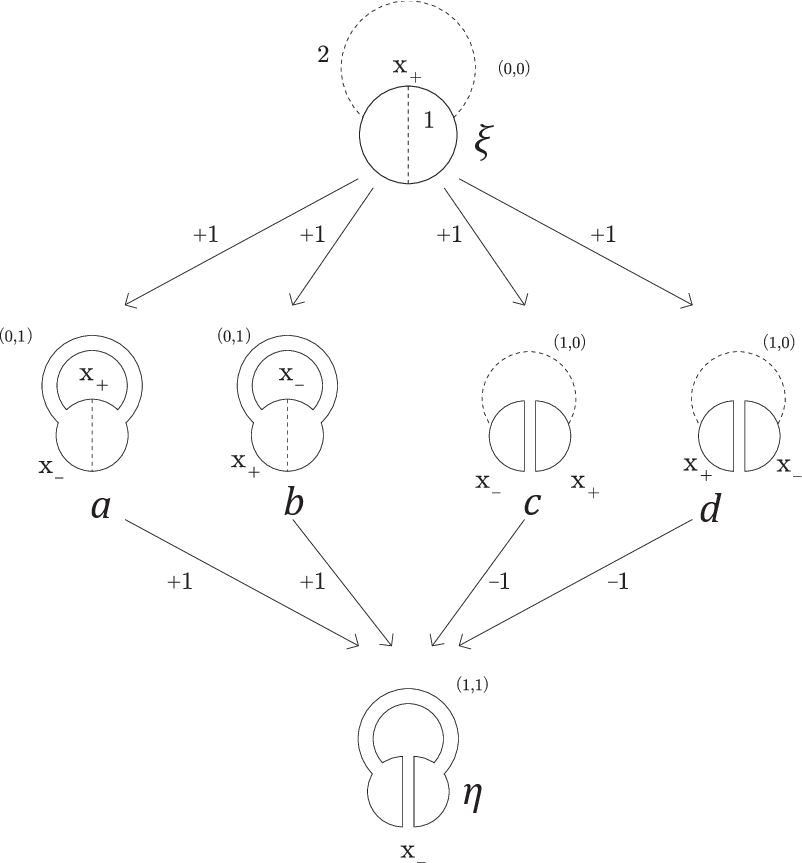}
\caption{
                 {\bf  The poset for
                 the decorated resolution configuration associated with
                                       a ladybug  configuration in $S^2$
                  }\label{uenp}
             }
\end{figure}

\bb
\subsection{Ladybug 
and quasi-ladybug configurations for link diagrams in surfaces}\label{hiyoko}
\bb

\begin{definition}({\bf\cite{KauffmanNikonovOgasa}})\label{LQ}
Let $D$ be a resolution configuration
which is made of one circle and two m-arcs (multiplication arc).  \\

Stand at a point in the circle where you see an arc to your right. 
Go ahead along the circle. Go around one time.
Assume that you encounter the following pattern:
 In the order of travel you next touch the other arc.
Then you touch the  first arc.
Then you touch the other arc again.
Finally, you come back to the point at the beginning. \\

Since both arcs are m-arcs,
both satisfy the following property:
At both endpoints of each arc,
you see the arc in the same side -- either on
the right hand side and on the left hand side.\\

If you see the arcs
both in the right hand side
and in the left hand side
(respectively, only in the right 
hand side)
while you go around one time,
we call $D$ a {\it ladybug configuration}
(respectively,
 {\it quasi-ladybug configuration}).

If $F$ is the 2-sphere,
our definition of ladybug configurations
is the same as
that
in in \S\ref{tamago}. 

\bb
Let $D$ be
a ladybug (respectively, quasi-ladybug) configuration.
Then $Z(D)$ have only one circle and $A(D)$ have only two arcs.
Make $s(D)$.
Give $D$ (respectively, $s(D)$) a labeling $x$  (respectively, $y$).
We call the decorated resolution configuration $(D,y,x)$
a {\it decorated resolution configuration associated with
the ladybug $($respectively, quasi-ladybug$)$ configuration $D$}.
\end{definition}

 Note that  $(D,y,x)$ may be empty 
 as explained below.

Since each of $D$ and $s(D)$ has only one circle,
we can let $x_+$ or $x_-$ denote $x$ (respectively, $y$).

See Figure \ref{quasiT2}.

The partial order defined in
\cite[Definition 2.10]{LSk} is defined in the case of link diagrams in $S^2$.
The authors \cite{KauffmanNikonovOgasa} generalized it to the case of links in thickened surfaces.

\begin{figure}
\includegraphics[width=0.7\textwidth]{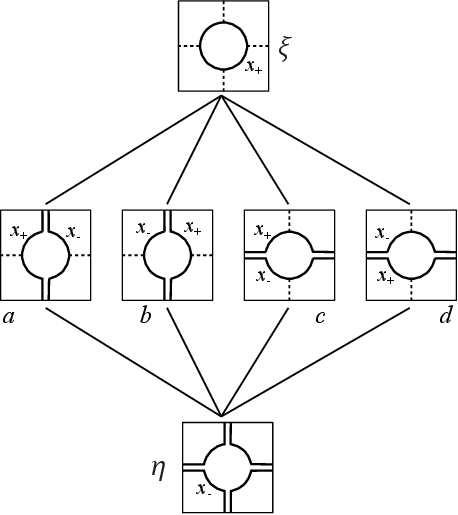}
\caption{{\bf
The poset for
a decorated resolution configuration $(D,x_-,x_+)$ associated with
a quasi-ladybug configuration on $T^2$:
We envelope $T^2$ along two circles as usual,
and draw six labeled resolution configurations.
Here, we have $[\xi;a]\cdot[a;\eta]=[\xi;b]\cdot[b;\eta]=-[\xi;c]\cdot[c;\eta]=-[\xi;d]\cdot[d;\eta].$
}\label{quasiT2}}
\end{figure}

\begin{proposition}{\bf(\cite{KauffmanNikonovOgasa})}\label{daijida}
\h$(1)$
Let $D$ be  a quasi-ladybug configuration  in a surface $F$.
Assume that the only one circle in $D$ is contractible.
Let $F$ be the torus.
Then there is a non-vacuous decorated resolution configuration
$(D, x_-,x_+)$ associated with $D$.

\bs
\h$(2)$
Let $D$ be  a quasi-ladybug configuration  in a surface $F$.
Assume that the only one circle in $D$ is contractible.
Let $(D,y,x)$ be a decorated resolution configuration  associated with $D$.
Assume that  the genus of $F$ is greater than one.
Then $(D,y,x)$ is empty 
for arbitrary $x$ and $y$.

\bs
\h$(3)$
 Let $D$ be a ladybug $($respectively, quasi-ladybug$)$ configuration in a surface $F$.
 Let $(D,y,x)$ be a decorated resolution configuration  associated with $D$.
Assume that the only one circle in $D$ is non-contractible.
Then $(D,y,x)$ is empty 
for arbitrary $x$ and $y$.

\bs
\h$(4)$
Let $F$ be an arbitrary surface.
There is a ladybug configuration $D$ in $F$
such that
a decorated resolution configuration $(D,y,x)$ associated with $D$
is non-empty. 
\end{proposition}

\subsection{Decorated resolution configurations of index $3$}

All resolution configurations with three arcs are made from
graphs in Figure \ref{fig:cases3}.

{\normalsize
\begin{figure}
\centering
\targ{4}{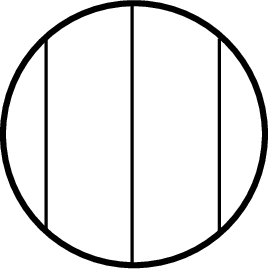}{1} \quad \targ{4}{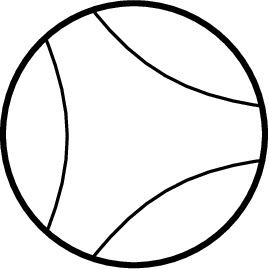}{2} \quad \targ{4}{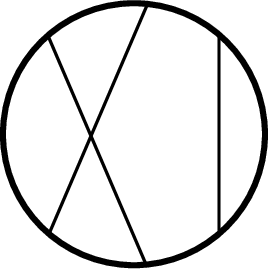}{3} \quad \targ{4}{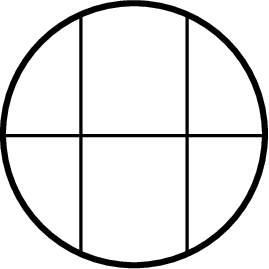}{4} \quad \targ{4}{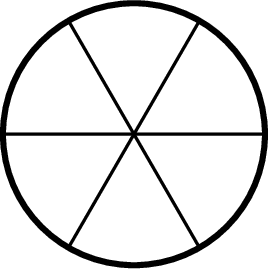}{5} \\
\targ{3.5}{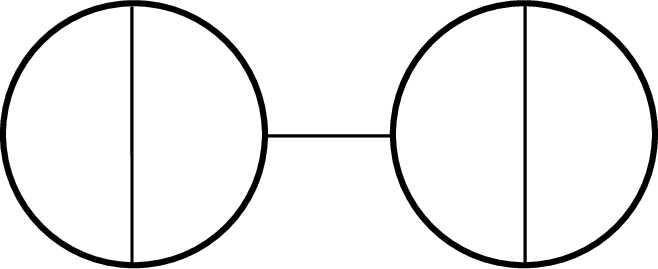}{6} \quad \targ{3.5}{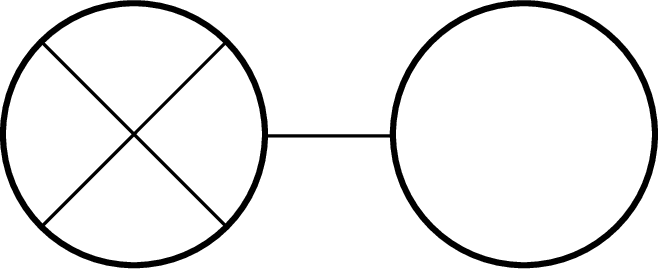}{7} \quad \targ{3.5}{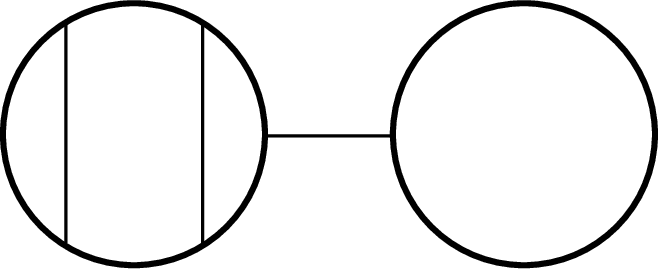}{8} \\
\targ{3.5}{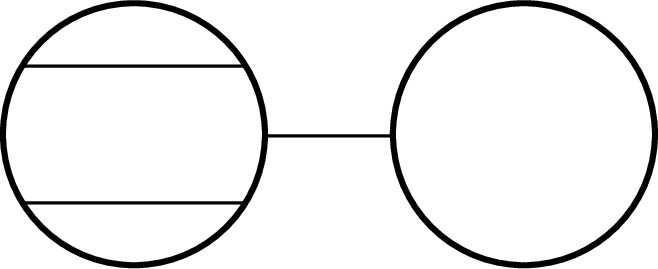}{9}\quad \targ{3.5}{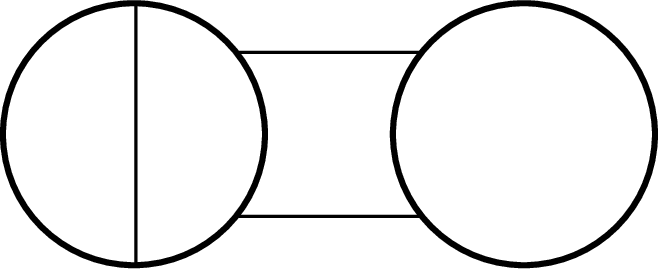}{10}\quad \targ{3.5}{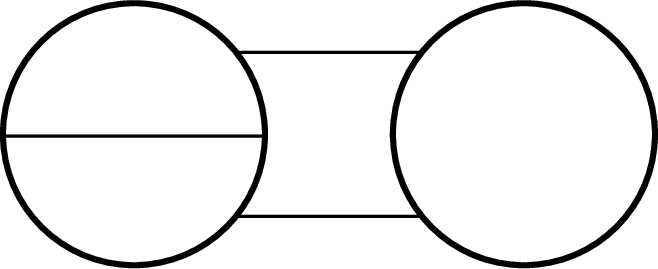}{11} \\
\targ{3.5}{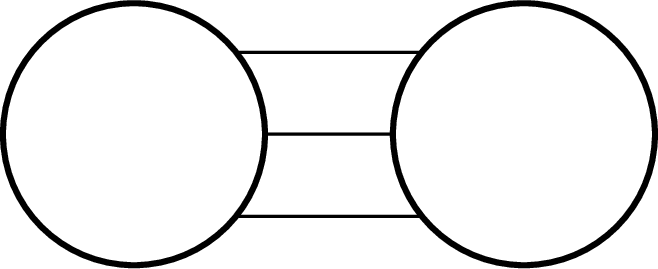}{12} \quad \targ{4}{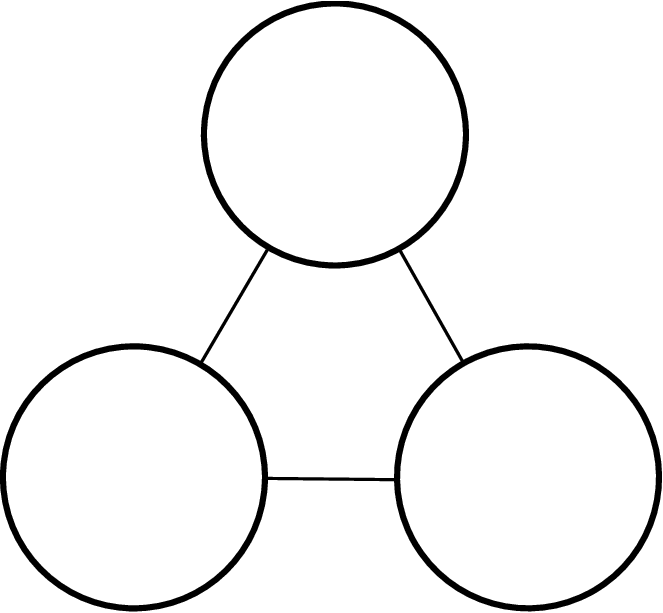}{13} \quad \targ{3.5}{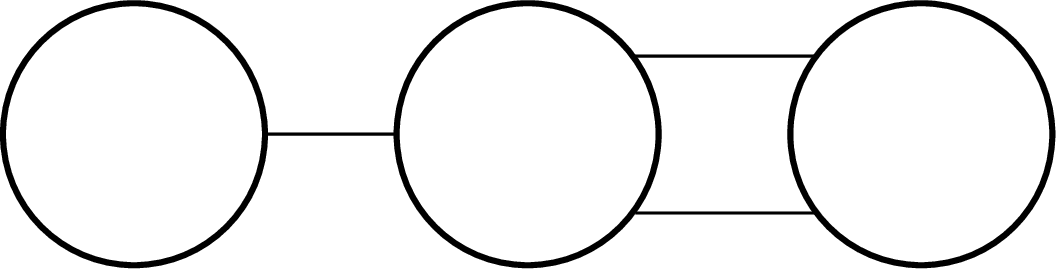}{14} \\
\targ{3.5}{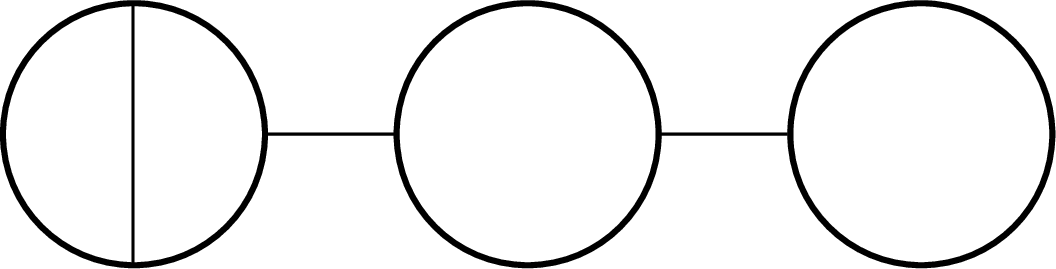}{15}\quad \targ{3.5}{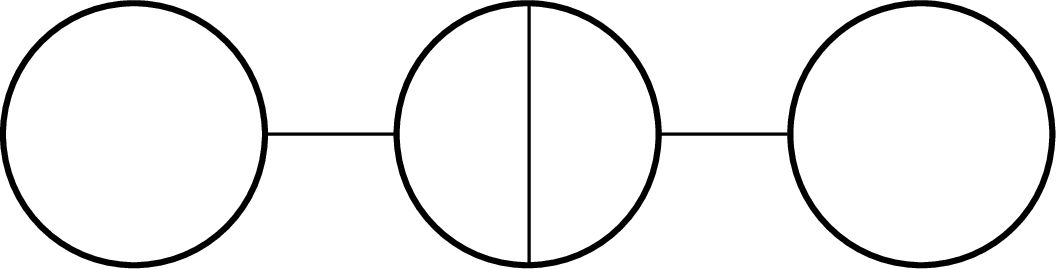}{16} \\ \targ{3}{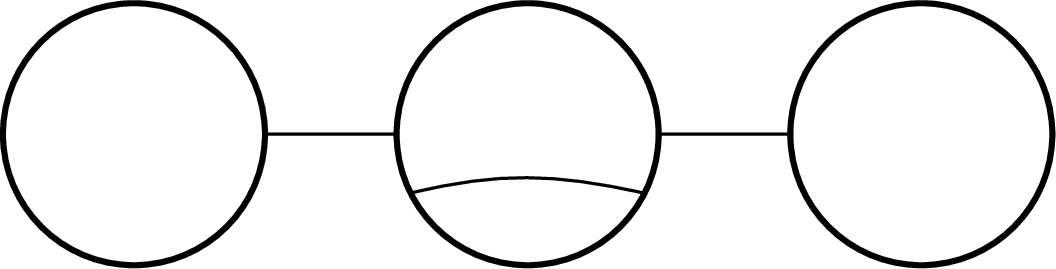}{17} \quad \targ{3}{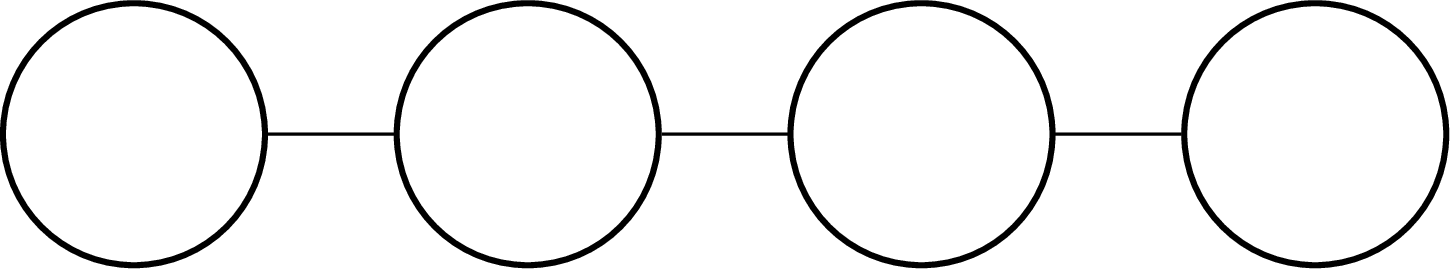}{18} \\ \targ{7}{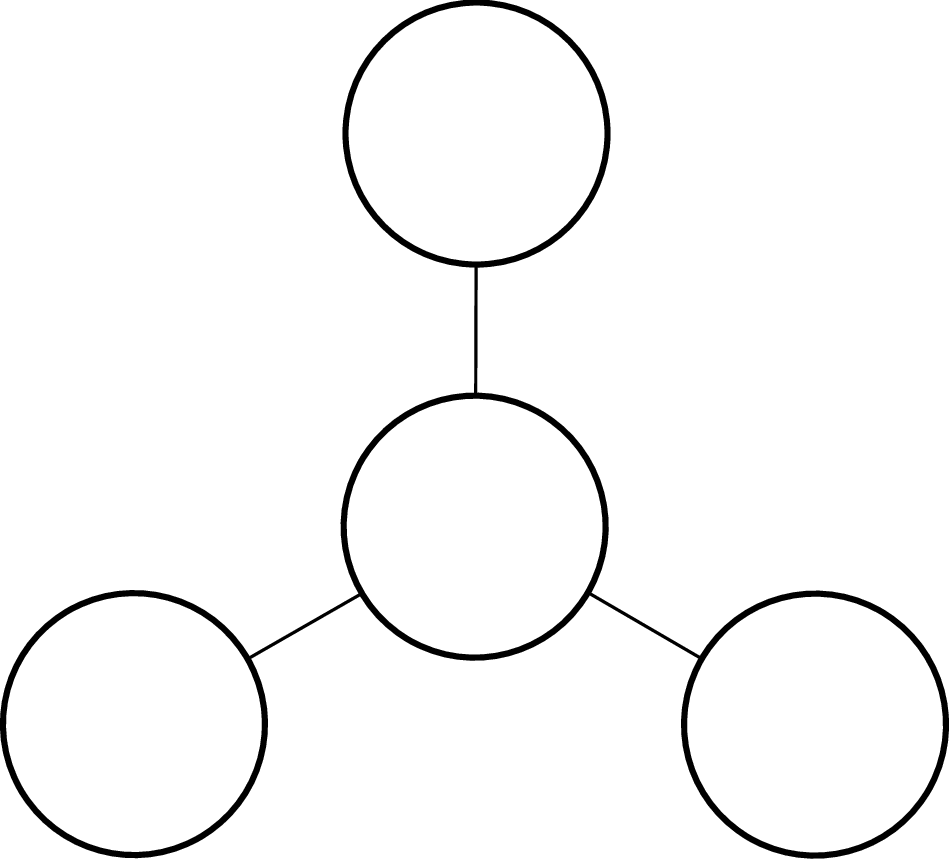}{19}
\caption{{\bf
Connected graphs of the resolution configurations of index 3:
The segments denote arcs. We do not use dotted segments here.
}}\label{fig:cases3}
\end{figure}
}

Most of the configurations contain a leaf or a coleaf. Hence, they can be reduced to decorated configurations of smaller index.

Using Propositions~\ref{prop:1circle_multiplicity} and~\ref{prop:kcircle_multiplicity} we can enumerate the diagrams without leaves and coleaves, see Fig.~\ref{fig:cases}.

\begin{figure}[h]
\includegraphics[width=140mm]{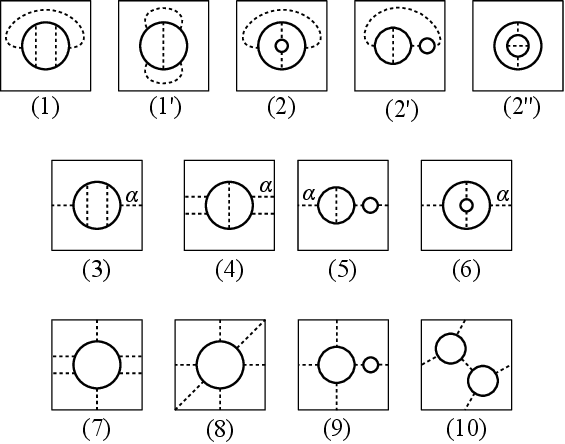}
\caption{{\bf Initial configurations of decorated resolutions of index $3$ with one circle and without leaves and coleaves. The labels of the circles are $x_+$.}}\label{fig:cases}
\end{figure}

The diagrams (1)--(2'') are local, i.e. can be drawn in a disk of the surface. Note that the diagrams (1) and (1') ((2), (2') and (2'')) are isotopic if one considers them as diagrams in the sphere $S^2$. Then diagrams (3)--(6) are parameterized by the homotopy class $\alpha$ of a nontrivial simple loop in the sphere. The diagrams (7)--(8) contain a pair of interlaced outer chords (quasi-ladybug configuration) and can occur only when $F=T^2$ is the torus.

\section{Moduli spaces}\label{mod}

We prove Theorem \ref{thm:moduli_system_existence} in this section.

We define the moduli spaces $\MM(D, x, y)$ by induction on the index $n$ of decorated resolution configuration $(D, x, y)$.

\subsection{Case $n=1$}\label{subseichi}
 We set $\MM(D, x, y)\simeq \MM_{\CC_C(1)}(\bar 1, \bar 0)$ to be one point.

\subsection{Case $n=2$}\label{subseni}
 The moduli space $\MM_{\CC_C(2)}(\bar 1, \bar 0)$ can be identified with the segment $I=[0,1]$.

 In all cases except the ladybug and quasi-ladybug configurations (see Fig.~\ref{fig:ladybug_cases}) there are four labeled resolution configurations in $(D, x, y)$ that correspond to the vertices of a square, i.e. the objects of $\CC_C(2)$. Thus, we can identify the moduli space $\MM(D, x, y)$ with $\MM_{\CC_C(2)}(\bar 1, \bar 0)=I$.

 In a (quasi)-ladybug configuration the boundary $\partial\MM(D, x, y)$ consists of four points $a,b,c,d$ which correspond to the paths in the diagram of the decorated resolution configuration, see Fig.~\ref{fig:ladybug_right} and~\ref{fig:quasi-ladybug_pairs}. By induction, the points $a$ and $b$ project to one end of the segment $I$, and $c$ and $d$ project to the other end. We must extend this projection to 
the $2$-fold  covering 
over $I$. There are two ways to do this, and we must choose one of them.

\begin{figure}[h]
\centering\includegraphics[width=0.4\textwidth]{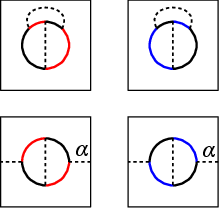}
\caption{Right pairs (right column) and left pairs (left column) in ladybug resolution configurations}\label{fig:ladybug_right_left_pairs}
\end{figure}

 For a ladybug configuration we can use left-right pair convention from the paper~\cite{LSk}. The ends of the arcs splits the cycle of the decorated resolution configuration into four segments. For the {\em right pairs}, we take the segments that start in the endpoint of the arc which goes to the right, and end in the endpoint of the arc which goes to the left, see Fig.~\ref{fig:ladybug_right_left_pairs}. The we pair the labeled resolution configurations which have the same labels on the distinguished segments of the cycle: $a$ with $d$, $b$ with $c$. The moduli space $\MM(D, x, y)$ is the disjoint union of segments $ad$ and $bc$, see Fig.~\ref{fig:ladybug_moduli_space}.

 Analogously, the moduli space for the left pairs can be defined.

    \begin{figure}[h]
\centering\includegraphics[width=0.5\textwidth]{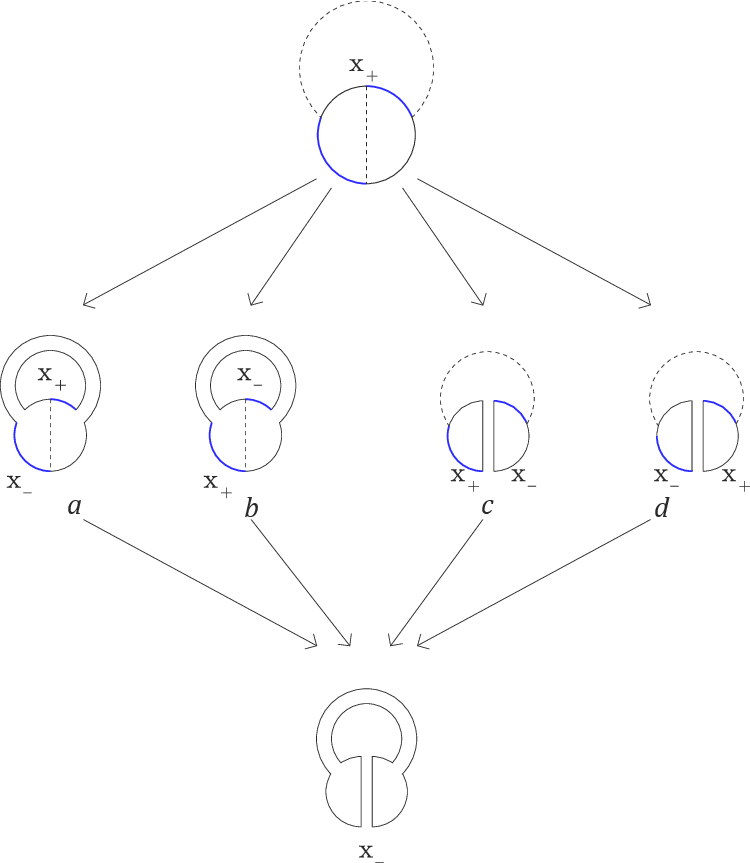}
\caption{Decorated diagram in a ladybug case. The blue segments are the right pairs in the cycle.}\label{fig:ladybug_right}
\end{figure}

\begin{figure}[h]
\centering\includegraphics[width=0.15\textwidth]{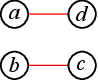}
\caption{The moduli space in (quasi-)ladybug case}\label{fig:ladybug_moduli_space}
\end{figure}

In the quasi-ladybug case, the cycle must be contractible (otherwise the decorated resolution configuration will be empty due to the homotopical grading). Then the orientation of the torus induces a canonical orientation of the cycle.

Any arc with the ends on the cycle determines a homology class in $H_1(T^2,\Z)$ (we take the class of the loop that the arc becomes after contraction the cycle to a point).

Let us define an analogue of right pairs in the quasi-ladybug case. 
Fix a prime element $\lambda\in H_1(T^2,\Z)$. \label{pageprime} 
It defines a simple curve in the torus.  Choose a class $\mu$ such that  $\lambda\cdot\mu=1$. Then $\lambda, \mu$ is a basis of $H_1(T^2,\Z)$.

Any arc $a$ with ends on the cycle determines a homology class in $H_1(T^2,\Z)$. Then $a=p\lambda+q\mu$, $p,q\in\Z$. We assign the number $-\frac pq$ to the arc $a$
 (if $a=\pm\lambda$ we assign $-\infty$ to $a$). \label{pagemaimuge}
 This numbering defines an order on the arcs of the decorated resolution configuration. Informally speaking, we number the arcs moving counterclockwise on the cycle, starting from an endpoint of the longitude on the cycle, see Fig.~\ref{fig:quasi-ladybug_order_pairs}.

In a quasi-ladybug decorated resolution configuration we take the segments of the cycle which start at the arc with the bigger number and end at the arc the smaller number, see Fig.~\ref{fig:quasi-ladybug_order_pairs}.

\begin{figure}[h]
\centering\includegraphics[width=0.2\textwidth]{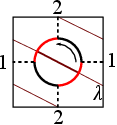}
\caption{The $\lambda$-pair.}\label{fig:quasi-ladybug_order_pairs}
\end{figure}

We shall call this pair of segments of the cycle the {\em $\lambda$-pair}. The other two segments form the {\em $\bar\lambda$-pair}, see Fig.~\ref{fig:quasi-ladybug_lambda_pairs}.

\begin{figure}[h]
\centering\includegraphics[width=0.4\textwidth]{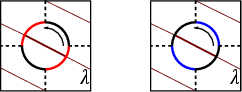}
\caption{The $\lambda$-pairs (left) and the $\bar\lambda$-pairs (right).}\label{fig:quasi-ladybug_lambda_pairs}
\end{figure}

Note that the $\lambda$-pair does not depend on the orientation of the cycle in the case when the arcs differ from $\lambda$ (as elements in $H_1(T^2,\Z)$). The segments which form the $\lambda$-pair can be thought of as the segments of the cycle (with ends at the arcs) which intersect the longitude $\lambda$. But the orientation matters when one of the arcs is homologous to $\lambda$, 
see Fig. 
\ref{fig:quasi-ladybug_order_pairs0}. 

\begin{figure}[h]
\centering\includegraphics[width=0.4\textwidth]{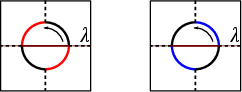}
\caption{The $\lambda$ and $\bar\lambda$-pairs when one of the arcs is homologous to $\lambda$.}\label{fig:quasi-ladybug_order_pairs0}
\end{figure}

We pair the labeled resolution configurations which have the same labels on the distinguished segments of the cycle ($a$ with $d$, $b$ with $c$), see Fig.~\ref{fig:quasi-ladybug_pairs}. Thus, we get the moduli space $\MM(D, x, y)$ to be equal the disjoint union of segments $ad$ and $bc$, see Fig.~\ref{fig:ladybug_moduli_space}.

   \begin{figure}[h]
\centering\includegraphics[width=0.5\textwidth]{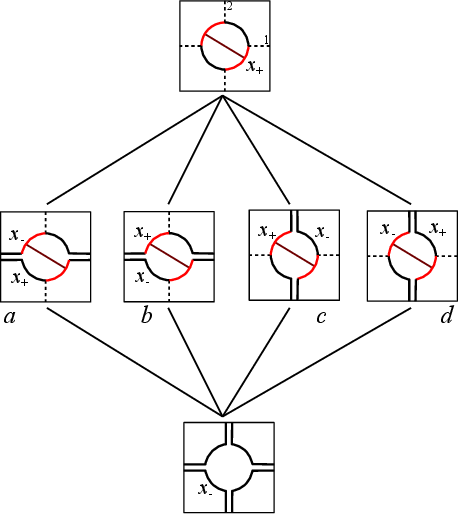}
\caption{Decorated diagram in the quasi-ladybug case. The red segments are the $\lambda$-pairs.}\label{fig:quasi-ladybug_pairs}
\end{figure}

Thus, in order to define the moduli spaces of index $2$ 
we need to fix choices $\pi(0),\pi(\alpha)\in\{l,r\}$ 
of left/right pairs for ladybug configurations $L_0$ and $L_\alpha$, 
and choices 
$\pi(\alpha,\beta)=\pi(\beta,\alpha)\in\{\lambda,\bar\lambda\}$ \label{pageonaji?}
between $\lambda$ and $\bar\lambda$-pairs for 
quasi-ladybug configurations $Q_{\alpha,\beta}$. 
This set of choices $\pi$ is called a {\em pairing}.

Let $\pi_{r,\lambda}$ \label{pagertte}   
%
denote the pairing such that $\pi_{r,\lambda}(0)=r$, $\pi_{r,\lambda}(\alpha)=r$ for all $\alpha$, and $\pi_{r,\lambda}(\alpha,\beta)=\lambda$ for all $\alpha,\beta$. Analogously, one defines the pairings $\pi_{l,\lambda}$, $\pi_{r,\bar\lambda}$, $\pi_{l,\bar\lambda}$. We will call these pairings {\em regular}.

From now on, we fix some pairing $\pi$ (regular or not) and define the moduli spaces $\MM(D, x, y)$ of index $2$ according to it.

\subsection{Case $n=3$}\label{subsesan}
The moduli space $\MM_{\CC_C(3)}(\bar 1, \bar 0)$ is a hexagon.

Let $(D, x, y)$ be a decorated resolution configuration of index $3$. The boundary $\partial\MM(D, x, y)$ is defined by induction, and we need to extend it to a moduli space $\MM(D, x, y)$ and a covering $\MM(D, x, y)\to \MM_{\CC_C(3)}(\bar 1,\bar 0)$.

If the decorated resolution configuration does not include a (quasi)-ladybug configuration, there is a bijection between the labeled resolutions of $(D, x, y)$ and the vertices of the $3$-cube, and we set $\MM(D, x, y)=\MM_{\CC_C(3)}(\bar 1,\bar 0)$.

If the decorated resolution configuration includes only a ladybug of type $L_0$, we are in the classical situation that was treated in~\cite{LSk}. The boundary $\partial\MM(D, x, y)$ is (the boundary of) two hexagons which project naturally to $\partial\MM_{\CC_C(3)}(\bar 1, \bar 0)$. We extend this projection to the trivial $2$-fold covering $\MM(D, x, y)\to \MM_{\CC_C(3)}(\bar 1,\bar 0)$.

If the decorated resolution configuration includes a ladybug of type $L_\alpha$, then it can not contain configurations of type $L_0$ or  $Q_{\alpha,\beta}$ (otherwise the decorated resolution configuration is empty because of homotopical grading). Then we have the following initial labeled resolution configurations, see Fig.~\ref{fig:cases} (3)--(6). In all cases the homotopical grading does not interferes the poset structure of the decorated resolution configuration. Hence, we can treat the decorated configuration as if it were planar. Thus, the boundary $\partial\MM(D, x, y)$ forms two hexagons as in the classical ladybug case, and the moduli space $\MM(D, x, y)$ is defined as the trivial $2$-fold covering space over $\MM_{\CC_C(3)}(\bar 1,\bar 0)$.



Let us consider the case when the decorated resolution configuration includes a quasi-ladybug configuration, 
see Fig.~\ref{fig:cases} (7)--(10).


In the diagrams 
(7), (9), a quasi-ladybug configuration of type $L_{\alpha,\beta}$ appears twice in the decorated resolution configuration. In these two cases the homotopical grading does not impose additional restrictions to the poset structure of the decorated resolution configuration. Then we can consider the horizontal (if $\pi(\alpha,\beta)=\lambda$) or vertical (if $\pi(\alpha,\beta)=\bar\lambda$) arcs as inner and work with the decorated configuration as with one including ladybug configuration of type $L_\beta$ (or $L_\alpha$). Thus, the moduli space $\MM(D, x, y)$ is the trivial $2$-fold covering over the hexagon $\MM_{\CC_C(3)}(\bar 1,\bar 0)$.

\begin{figure}[h]
\centering\includegraphics[width=0.12\textwidth]{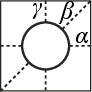}
\caption{Initial labeled resolution configuration for the decorated configuration $DQ(\alpha,\beta,\gamma)$}\label{fig:dodeca_quasiladybug}
\end{figure}

Let us consider the diagram in Fig.~\ref{fig:cases} (8). Denote the homology classes of the arcs by $\alpha$, $\beta$, $\gamma$, see Fig.~\ref{fig:dodeca_quasiladybug}. Then the decorated resolution configuration includes three quasi-ladybug configurations of type $Q_{\alpha,\beta}$, $Q_{\beta,\gamma}$ and $Q_{\alpha,\gamma}$, see Fig.~\ref{fig:case_8l}.

\begin{figure}[h]
\centering\includegraphics[width=0.7\textwidth]{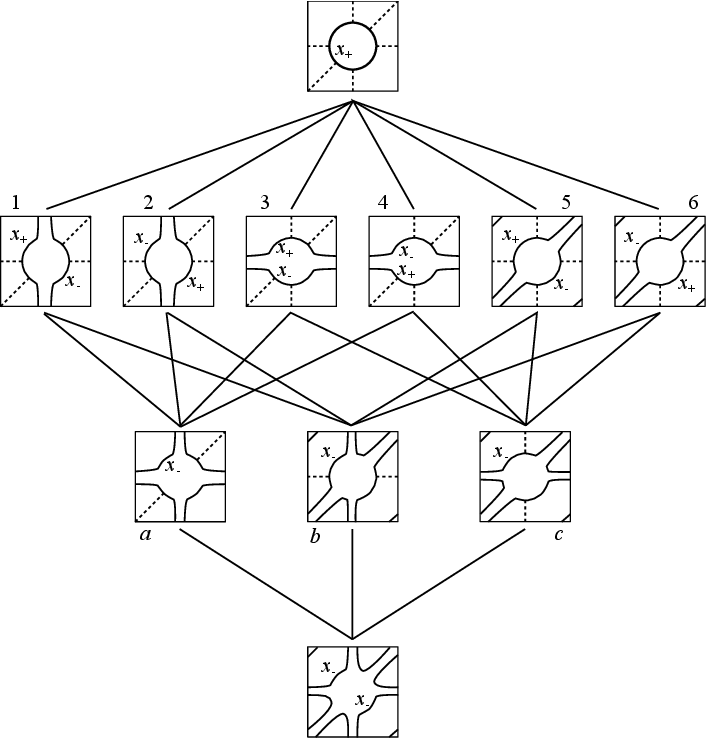}
\caption{The decorated diagram $DQ(\alpha,\beta,\gamma)$ of index $3$}\label{fig:case_8l}
\end{figure}

The boundary $\partial\MM(D, x, y)$ consists of $12$ vertices which corresponds to paths from the initial to the final labeled resolution configuration in the diagram of the decorated configuration in Fig.~\ref{fig:case_8l}. Any path is determined by the intermediate labeled configuration, for example, $1a$, $4c$ etc. An edge of $\partial\MM(D, x, y)$ corresponds to switching between two paths with a common edge, see Fig.~\ref{case_8_moduli_space1}.
There are six fixed edges $1a-1b$, $2a-2b$, $3a-3c$, $4a-4c$, $5b-5c$, $6b-6c$. The other six edges of $\partial\MM(D, x, y)$ depend on the chosen pairing $\pi$.

\begin{figure}[h]
\includegraphics[width=0.4\textwidth]{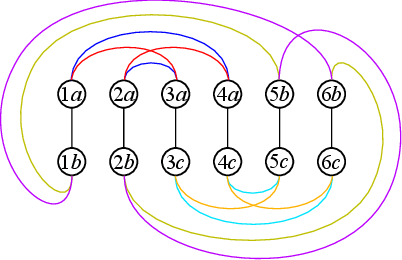}
\caption{Pairings in $\partial\MM(D, x, y)$. Red, orange and gold edges correspond to $\lambda$-pairs. Blue, azure and magenta labels correspond to $\bar\lambda$-pairs. The black edges are fixed.}\label{case_8_moduli_space1}
\end{figure}

Let us first consider the case when all pairs are $\lambda$-pairs: $\pi(\alpha,\beta)=\pi(\beta,\gamma)=\pi(\alpha,\gamma)=\lambda$.

Using the fixed longitude $\lambda$, we determine the $\lambda$-pairs for any pair of arcs, see Fig.~\ref{fig:case_8_order_pairs}.

   \begin{figure}[h]
\centering\includegraphics[width=0.25\textwidth]{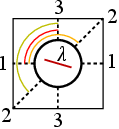}
\caption{Segments of the $\lambda$-pairs.}\label{fig:case_8_order_pairs}
\end{figure}

We use the $\lambda$-pairs to find the moduli space of the quasi-ladybug faces. The moduli space of the face containing the labeled resolution configuration $a$ consists of segments $1a-3a$ and $2a-4a$, the other two faces give the segments $1b-5b$, $2b-6b$, $3c-5c$ and $4c-6c$. Thus, the boundary $\partial\MM(D, x, y)$ forms two hexagons, see Fig.~\ref{fig:case_8_order_modulus}. Since the space $\partial\MM_{\CC_C(3)}(\bar 1,\bar 0)$ is a hexagon, the covering map on the boundary is trivial.

   \begin{figure}[h]
\centering\includegraphics[width=0.35\textwidth]{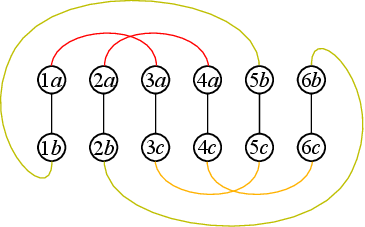}
\caption{The boundary $\partial\MM(D, x, y)$ of the moduli space for the $\lambda$-pairing.}\label{fig:case_8_order_modulus}
\end{figure}

\begin{figure}[h]
\centering\includegraphics[width=0.15\textwidth]{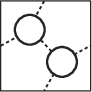}
\caption{Initial labeled resolution configuration for the decorated configuration $DQ'(\alpha,\beta,\gamma)$. The classes $\alpha,\beta,\gamma$ are the ones of loops that appear after contracting the two cycles to points and deleting one of the three arcs.}\label{fig:dodeca_quasiladybug1}
\end{figure}

Let us consider the diagram in Fig.~\ref{fig:dodeca_quasiladybug1}. This decorated resolution configuration $DQ'(\alpha,\beta,\gamma)$ (see Fig.~\ref{fig:case_8_dual_lambda}) is dual to the one considered above. Thus, it has the isomorphic moduli space $\MM(D, x, y)$: two hexagons if the number of $\bar\lambda$-pairings among $\pi(\alpha,\beta)$, $\pi(\beta,\gamma)$, $\pi(\alpha,\gamma)$ is even, and 
a 
dodecagon branched 
over the hexagon $\MM_{\CC_C(3)}(\bar 1, \bar 0)$ if the number $\bar\lambda$-pairings is odd.

Thus, we see that the regular pairings $\pi_{r,\lambda}$, $\pi_{l,\lambda}$ induce trivial coverings of the moduli spaces $\MM(D, x, y)\to\MM_{\CC_C(3)}(\bar 1, \bar 0)$ of index three.

\begin{figure}[h]
\centering\includegraphics[width=0.7\textwidth]{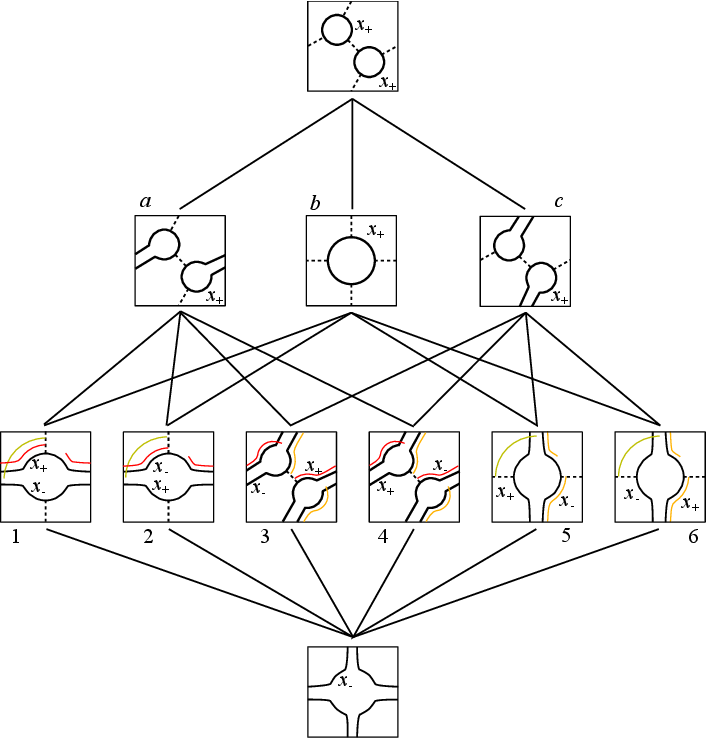}
\caption{The decorated diagram $DQ'(\alpha,\beta,\gamma)$ of index $3$}\label{fig:case_8_dual_lambda}
\end{figure}

\subsection{Case $n\geqq4$}\label{subseijo}

Since there is no obstruction to extension of trivial covering in higher dimensions (that is, $n>3$), the moduli spaces of regular $\lambda$-pairings are trivial coverings over the cubic moduli spaces.

Thus, we have proved Theorem \ref{thm:moduli_system_existence}.


\np
\section{The second Steenrod square operator}\label{sq}

\subsection{The first Steenrod square operator $Sq^1$ }\label{sqq1}%

\h
In \cite{Steenrod,SE}
the Steenrod square $Sq^* (*\in\Z)$ is defined.
Let $X$ and $X'$ be compact CW complexes.
Let $\{C_i\}_{i\in\Z}$  be a chain complex. Assume that $\{C_i\}_{i\in\Z}$ is associated with both a CW decomposition on $X$
and a CW decomposition  on $X'$.
It is well-known that $Sq^1(X)=Sq^1(X')$
(see e.g. \cite[Introduction]{LSs})
and that $Sq^2(X)$ and $Sq^2(X')$ are different in general
(see e.g. \cite{Seed}).

Therefore, $Sq^1$ is not informative as a link invariant. Let us pass to $Sq^2$.

\bigbreak
\subsection{The second Steenrod square operator $Sq^2$ }\label{sqq}

\h
We review the definition of the second Steenrod square~\cite{Steenrod, SE}.

\begin{definition}\label{secondsq}
Let $K_m$ be the Eilenberg--MacLane space $K(\Z_2,m)$ for any natural number $m>1$.
By definition, $K_m$ is connected and
 $\pi_i(K_m)\cong\Z_2$ (respectively, $0$)
if $i=m$ (respectively, $i\neq m$), where $ i\ge1$.
It is known that $H^{m+2}(K_m;\Z_2)\cong\Z_2$. Denote the generator of
$H^{m+2}(K;\Z_2)\cong\Z_2$ by $\xi$.

Let $X$ be a CW complex and $[X,K_m]$ be the set of all homotopy classes of continuous maps $X\to K_m$.
Then $[X,K_m]=H^m(X;\Z_2)$.

For an arbitrary element $x\in H^m(X;\Z_2)$, take a continuous map $f_x:X\to K_m$ which corresponds to the class $x$.
Define the {\em second Steenrod square} $Sq^2(x)$ of $x$ to be $f^\ast_x(\xi)\in H^{m+2}(X;\Z_2)$.
\end{definition}

This definition is reviewed and explained very well  in \cite[section 3.1]{LSs}.

We review an important property of the second Steenrod square operator $Sq^2$, below.

\begin{proposition}\label{sq2}{\rm\bf(\cite[section 12]{Steenrod}.)}
Let $Y$ be any compact CW complex.
Let $Y^{(*)}$ be the $*$-skeleton of $Y (*\in\Z)$.
Let $m\in\Z$. Then the second Steenrod square 
$Sq^2(Y):H^{m}(Y;\Z_2)\to H^{m+2}(Y;\Z_2)$
is determined by the homotopy type of $Y^{(m+2)}/Y^{(m-1)}$.
\end{proposition}

This proposition is reviewed and explained very well  in \cite[section 3.1]{LSs}.

\subsection{The second Steenrod  square $Sq^2$ for links in the thickened torus
}\label{vitaC}


\h
Let $\mathcal L$ be a link in the thickened torus.
Make a
 Khovanov-Liphitzs-Sarkar stable homotopy type
for a triple of $\mathcal L$,  a degree 1 homology class $\lambda\in H_1(T^2;\Z_2)$
and the right-left choice.
Its second Steenrod square gives 
an isotopy invariant of $\mathcal L$.

Note: There are infinitely many choices of degree 1 homology classes.
However,
when we are given two link diagrams
and we compare the two,
we only have to calculate the second
Steenrod square in a finite cases.
We put emphasis on 
the fact that our invariant, the set of  Steenrod squares,
is calculable.
\\

In the case of links in $S^3$,
in \cite{LSs} Lipshitz and Sarkar showed a way to calculate 
$Sq^2$
by using classical link diagrams. 
Seed \cite{Seed} calculated the second Steenrod  square for links in $S^3$ by making a computer program of their method in \cite{LSs}. He found the following explicit pair.

\begin{theorem}\label{Seedrei}  {\bf (\cite{Seed})}
There are links $\mathcal J$ and $\mathcal J'$ in $S^3$ such that
the Khovanov homologies are the same, but such that the second Steenrod squares are different.

Therefore
there are links $\mathcal J$ and $\mathcal J'$ in $S^3$ such that
the Khovanov homologies are the same, but the Khovanov stable homotopy types are different.
\end{theorem}


It is very natural to ask the following question. Are there a pair of links in the thickened torus such that the homotopical Khovanov homologies are the same, but such that the second Steenrod squares are different?

Note that all links in $B^3$ are regarded as links in the thickened torus if we regard $B^3$ is embedded in the thickened torus. 
By 
Theorem \ref{Seedrei},
the above question has an affirmative answer.
\label{proof:homotopy_type_stronger_homology} 

Furthermore it is very natural to ask the following question. Are there a pair of links in the thickened torus which are not embedded in $B^3$ such that the homotopical Khovanov homologies are the same, but such that the second Steenrod squares are different?
 The answer is a main result.
See the following proof.

\bb

\h{\bf Proof of Main Theorem \ref{main}.}
Let $L$ be a link in the thickened torus. Consider the moduli system $\mathcal M$ on the torus constructed in Section \ref{mod}. 
Then ${\mathcal X}_\MM (L)$ is a Khovanov homotopy type with cubic moduli spaces which proves Main Theorem \ref{main}.(1).

An easy example which proves Main Theorem \ref{main}.(2) is given above in the page~\pageref{proof:homotopy_type_stronger_homology}.
We show a little more complicated example below, which is an alternative proof of Main Theorem \ref{main}.(2).

Let $C$ be a  circle in $T^2$ which represents
a
nontrivial
element of
$H_1(F;\Z)$.
Regard $A$ as a knot in $F\x[-1,1]$.
Take
$\mathcal J$ and $\mathcal J'$
in a 3-ball $B$ embedded in $F\x[-1,1]$,
which are written in Theorem \ref{Seedrei}.
Assume that $C\cap B=\emptyset$.
Make a disjoint 2-component link
which is made from $C$ and
$\mathcal {J}$ (respectively, $\mathcal {J'}$).
By Theorem \ref{Seedrei}, 
these two links have different Steenrod squares and the same Khovanov homology.  
See \cite[\S10.2]{LSk} for Khovanov-Lipshitz-Sarkar stable homotopy type
of disjoint links.

In this case, we do not have a quasi-ladybug configuration.
The right pair and the left one of ladybug situations
give the same Steenrod second square
by explicit calculus which uses that about the classical link diagram $K$.
(Note that the Steenrod square is only one element in this case.)
\qed\\

The above example is just the beginning of many possible applications of the result in this paper.
Further applications require deeper computations of the virtual Khovanov homology
and will be the subject of a subsequent paper.\\



\np
\noindent
Louis H. Kauffman

\noindent
Department of Mathematics, Statistics and Computer Science

\noindent
University of Illinois at Chicago

\noindent
851 South Morgan Street

\noindent
Chicago, Illinois 60607-7045

\noindent
USA






\noindent
kauffman@uic.edu
\\

\h Igor Mikhailovich Nikonov

\h Department of Mechanics and Mathematics

\h Lomonosov Moscow State University

\h  Leninskiye Gory, GSP-1

\h Moscow, 119991

\h Russia

\h nikonov@mech.math.msu.su
\\

\noindent
Eiji Ogasa

\noindent
Meijigakuin University, Computer Science

\noindent
Yokohama, Kanagawa, 244-8539

\noindent
Japan

\noindent
pqr100pqr100@yahoo.co.jp

\noindent
ogasa@mail1.meijigkakuin.ac.jp

}


\begin{thebibliography}{ABCD}
{\huge
\bibitem{APS}
M. M. Asaeda, J. H. Przytycki, and A. S. Sikora;
Categorification of the Kauffman bracket skein
module of $I$-bundles over surfaces,
{\it Algebraic \& Geometric Topology}
4 (2004) 1177–1210 ATG




\bibitem{B}
D. Bar-Natan:
On Khovanov's categorification of the Jones polynomial,
{\it Algebr. Geom. Topol.} 2(2002), 337–370 (electronic). MR 1917056 (2003h:57014).


\bibitem{Drobotukhina}
Yu.V. Drobotukhina: 
An analogue of the Jones polynomial for links in $\R P^3$ and a generalization of the Kauffman-Murasugi theorem {\it Leningrad Math. J.} (1991) 613-630.


\bibitem{Bourgoin}
M. O. Bourgoin: 
Twisted Link Theory, 
arXiv:math/0608233  (2006)





\bibitem{DKK} 
H A Dye, A Kaestner, and L H Kauffman:  
Khovanov Homology, Lee Homology and a Rasmussen Invariant for Virtual Knots, 
{\it Journal of Knot Theory and Its Ramifications} 26 (2017).




\bibitem{Jones} 
 V. F. R. Jones: Hecke Algebra representations of braid groups and link   
{\it Ann. of Math.} 126 (1987) 335-388. 




\bibitem{Kauffmanstate}
L. H. Kauffman:
State models and the Jones polynomial,
{\it Topology}  26 (1987) 395-407.




\bibitem{Kauffman1} 
L. H. Kauffman: 
Talks at MSRI Meeting in January 1997, AMS Meeting at University of Maryland, College Park in March 1997, Isaac Newton Institute Lecture in November 1997, Knots in Hellas Meeting in Delphi, Greece in July 1998, APCTP-NANKAI Symposium on Yang-Baxter Systems, Non-Linear Models and Applications at Seoul, Korea in October 1998

\bibitem{Kauffman} 
L. H. Kauffman: 
Virtual Knot Theory, 
{\it Europ. J. Combinatorics} (1999) 20, 663–691, 
{\it Article No. eujc}.1999.0314, 
{\it Available online at} http://www.idealibrary.com 
math/9811028 [math.GT].


\bibitem{Kauffmani} 
L. H. Kauffman: 
Introduction to virtual knot theory,  
{\it J. Knot Theory Ramifications} 21 (2012), no. 13, 1240007, 37 pp.







\bibitem{KauffmanNikonovOgasa} L. H. Kauffman, I. M. Nikonov,  and  E. Ogasa:  Khovanov-Lipshitz-Sarkar homotopy type for links in thickened higher genus surfaces   arXiv: 2007.09241[math.GT].





\bibitem{K}   M. Khovanov:  A categorification of the Jones polynomial,
{\it Duke Math. J.} 101 (2000), no. 3, 359–426. MR 1740682 (2002j:57025).



\bibitem{LSk} R. Lipshitz and S. Sarkar: A Khovanov stable homotopy type,
{\it J. Amer. Math. Soc.} 27 (2014), no. 4, 983–1042. MR 3230817



\bibitem{LSs}
R. Lipshitz and S. Sarkar:
A Steenrod square on Khovanov homology,
{\it J. Topol.} 7 (2014), no. 3, 817–848. MR 3252965








\bibitem{Man}
V O Manturov: 
 Khovanov homology for virtual links with arbitrary coeficients, 
{\it Journal of Knot Theory and Its Ramifications} 16 (2007), 
arXiv:math/0601152.


\bibitem{MN}
V. O. Manturov and I. M. Nikonov:
Homotopical Khovanov homology,
{\it Journal of Knot Theory and Its Ramifications}
 24 (2015) 1541003.

\bibitem{Igor}
 I. M. Nikonov: 
Virtual index cocycles and invariants of virtual links, 
arXiv:2011.00248



\bibitem{RT} 
N. Reshetikhin and V. G. Turaev: 
Invariants of 3-manifolds via link polynomials and quantum groups, 
{\it Inventiones mathematicae} 103 (1991) 547–597. 

\bibitem{Ru}  W. Rushworth: Doubled Khovanov Homology, {\it Can. J. Math.} 70 (2018) 1130-1172. 


\bibitem{Seed}  C. Seed:
Computations of the Lipshitz-Sarkar Steenrod square on Khovanov homology,
arXiv:1210.1882.


\bibitem{Steenrod}
N. E. Steenrod:
Cohomology operations, and obstructionsto extending continuous functions,
{\it Advances in Math. 8} (1972)  371-416.

\bibitem{SE}
 N. E. Steenrod (Author) and D. B. A. Epstein (Editor):
Cohomology Operations,
{\it Annals of Mathematics Studies, Princeton University Press}  (1962).


\bibitem{Tub}
D. Tubbenhauer: 
Virtual Khovanov homology using cobordisms, 
{\it J. Knot Theory Ramifications} 23 (2014), no. 9, 1450046, 91 pp. 

\bibitem{Viro} Viro:
Khovanov homology of Signed diagrams 2006 (an unpublished note).

\bibitem{W} E. Witten:  Quantum field theory and the Jones polynomial
{\it Comm. Math. Phys.  } 121 (1989) 351-399. 


}
\end{thebibliography}
\end{document}